\documentclass{article}

\usepackage{amsthm,amsfonts,amsmath,amssymb}
\usepackage{mathtools}
\usepackage{enumitem}

\usepackage{pgf,tikz}
\usetikzlibrary{arrows}
\usepackage{graphicx}
\usepackage{color}
\usepackage[all]{xy}

\newcommand{\Xd}{X}
\newcommand{\Li}{\mathcal{L}}
\newcommand{\LL}{\lvert \Li \rvert}
\newcommand{\LK}{\lvert K_{\Xd}\otimes\Li \rvert}

\newcommand{\D}{\mathcal{D}}
\newcommand{\OO}{\mathcal{O}}
\newcommand{\C}{\mathbb{C}}
\newcommand{\R}{\mathbb{R}}
\newcommand{\Z}{\mathbb{Z}}
\newcommand{\N}{\mathbb{N}}

\newcommand{\G}{\mathcal{G}}

\newcommand{\MCG}{MCG}

\newcommand{\LS}{\mathcal{S}}
\newcommand{\CC}{\mathcal{C}}
\newcommand{\mmu}{[\mu]}
\newcommand{\cp}[1]{\C P^{#1}}

\newcommand{\I}{\mathcal{I}}

\newcommand{\da}{\Delta_a}
\newcommand{\ak}{\alpha_{\kappa}}
\newcommand{\akp}{\alpha_{\kappa}'}
\newcommand{\ok}{\omega_{\kappa}}
\newcommand{\okp}{\omega_{\kappa}'}

\newcommand{\PSec}{\vert \Li \vert}
\newcommand{\Sec}{H^0 (X,\Li)}
\newcommand{\sh}{simple Harnack }
\newcommand{\ph}{phase-tropical }
\newcommand{\Hd}{\mathcal{H}_\Delta}
\newcommand{\A}{\mathcal{A}}
\newcommand{\ttor}{(\C^\ast)^2}
\usepackage{color}

\DeclareMathOperator{\Log}{Log}

\DeclareMathOperator{\im}{im}
\DeclareMathOperator{\id}{id}
\DeclareMathOperator{\spaut}{Sp}
\DeclareMathOperator{\di}{d}
\DeclareMathOperator{\itr}{int}

\newtheorem{Proposition}{Proposition}[section]
\newtheorem{Definition}[Proposition]{Definition}
\newtheorem{Lemma}[Proposition]{Lemma}
\newtheorem{Remark}[Proposition]{Remark}
\newtheorem{Theorem}{Theorem}
\newtheorem{Conjecture}{Conjecture}
\newtheorem{Corollary}[Proposition]{Corollary}
\newtheorem{Question}{Question}

\newtheorem{Notation}{Notation}
\newtheorem*{ack}{Acknowledgement}

\begin{document}

\title{The vanishing cycles of curves in toric surfaces I}
\author{R\'{e}mi Cr\'{e}tois and Lionel Lang}

\maketitle

\begin{abstract}
 This article is the first in a series of two in which we study the vanishing cycles of curves in toric surfaces. We give a list of possible obstructions to contract vanishing cycles within a given complete linear system. Using tropical means, we show that any non-separating simple closed curve is a vanishing cycle whenever none of the listed obstructions appears. 
\end{abstract}

\footnote{{\bf MSC-2010 Classification:} 14T05, 14M25, 32S30}

\section{Introduction}

Along a degeneration of a smooth algebraic curve $C$ to a nodal curve, there is a simple closed curve in $C$ that gets contracted to the nodal point. This simple closed curve, well defined up to isotopy, is called the vanishing cycle of the degeneration.

If the question of contracting simple closed curves abstractly in $\overline{\mathcal{M}}_g$ leaves no mystery, the embedded case is still open. In particular, we may ask what are the possible vanishing cycles of a curve $C$ embedded in a toric surface $\Xd$, where the underlying degenerations are constrained in a fixed complete linear system $\LL$ on $\Xd$.  As simply as it is stated, this question, suggested by Donaldson (\cite{Do}), admits no definitive answer. Another approach is to study the monodromy representation of the fundamental group of the complement of the discriminant $\D \subset \LL$ inside the mapping class group $\MCG(C)$. The study of such fundamental groups was suggested by Dolgachev and Libgober (\cite{DL}) and is still an active field of investigation. From a wider perspective, the description of the monodromy map gives an insight on the universal map $\LL \setminus \mathcal{D} \rightarrow \mathcal{M}_g$ which is far from being understood (see $e.g$ \cite{CV}).

We know from \cite{Wajn} that for a smooth plane curve $C$ of genus $g$, we can contract $2g$ simple closed curves whose complement in $C$ is a disc. It leads in particular to the topological classification of smooth projective surfaces in $\cp{3}$ given in \cite{MM}.
In \cite{Beau}, Beauville determines the image of the algebraic monodromy map for hypersurfaces of degree $d$ in $\cp{m}$, for any $m$ and $d$.

In the present paper, we focus on the case of curves but investigate any complete linear system $\LL$ on any smooth toric surface $\Xd$. We point out the obstructions to contract cycles in the curve $C$. They mainly arise from  the roots of the adjoint line bundle $K_\Xd \otimes \Li$. Square roots of $K_\Xd \otimes \Li$ were already considered in \cite{Beau}. The obstructions provided by higher order roots of $K_\Xd \otimes \Li$ were studied in \cite{earlesipe} and \cite{sipe} but were somehow forgotten in the common literature until recently. In particular, they do not show up in the work of \cite{Beau} as the obstructions they provide are undetectable at the homological level.
It might also happen that all the curves in $\LL$ are hyperelliptic. In such case, the monodromy map has to preserve the hyperelliptic involution.

We show that there are no other obstructions than the ones mentioned above.

\begin{Theorem}
Let $\Xd$ be a smooth complete toric surface, $\Li$ an ample line bundle on $\Xd$ and $C \in \LL$ a generic smooth curve. Then, the monodromy map 
\[ \mu : \pi_1(\LL \setminus \D) \rightarrow \MCG(C)\]
is surjective if and only if $C$ is not hyperelliptic and if the adjoint line bundle of $\Li$ admits no root.
In this case, any non-separating simple closed curve in $C$ is a vanishing cycle.
\end{Theorem}

Let us mention a very recent independent work by Salter (\cite{Sal}) in which he proves in particular the result above in the case of $\cp{2}$.

As suggested by \cite{Beau}, some of the latter obstructions persist in homology. In the statement below, $\spaut(H_1(C,\Z))$ denotes the group of automorphisms preserving the intersection form.

\begin{Theorem}
Let $\Xd$ be a smooth complete toric surface, $\Li$ an ample line bundle on $\Xd$ and $C \in \LL$ a generic smooth curve. Then, the algebraic monodromy map 
\[ \mmu : \pi_1(\LL \setminus \D) \rightarrow \spaut(H_1(C,\Z))\]
is surjective if and only if $C$ is not hyperelliptic and if the adjoint line bundle of $\Li$ admits no root of even order.
In this case, any non-separating simple closed curve in $C$ is homologous to a vanishing cycle.
\end{Theorem}

\setcounter{Theorem}{0}

We introduce new techniques in order to determine the image of the monodromy map $\mu$. We first use the intensively studied simple Harnack curves (see \cite{Mikh},  \cite{KO}), for which all possible degenerations are known. From them, we then construct loops in $\LL \setminus \mathcal{D}$ with explicit monodromy in $\MCG(C)$ using tropical methods. To this aim, we consider partial phase-tropical compactifications of $\mathcal{M}_g$. According to \cite{L}, we have a proper understanding of the closure of $\LL \setminus \mathcal{D}$ in this compactification in terms of Fenchel-Nielsen coordinates. Using these coordinates, we construct loops with prescribed monodromy in the closure. We then push these loops back to $\LL \setminus \mathcal{D}$ with the help of Mikhalkin's approximation Theorem, see Theorem \ref{thm:rea}. The efficiency of such a technique suggest that the discriminant intersects nicely the boundary of a globally well defined \ph compactification of $\mathcal{M}_g$. It also has the advantage of being explicit in the sense that we can actually draw the cycles that we can contract on our reference curve $C$. 

In \cite{article2}, we will use this technique further to determine the image of the (algebraic) monodromy in the hyperelliptic and Spin cases. Motivated by these results and the constructions given in section \ref{sec:conj}, we give a conjecture about the image of $\mu$ in the general case (see Conjecture \ref{conj:general}). In \cite{Sal}, Salter proves this conjecture for degree $5$ curves in $\cp{2}$. Coming back to the broader question asked in \cite{Do}, we believe the tropical approach could be carried to the case of smooth surfaces in toric $3$-folds. Another interesting and fundamental variation of the problem is the study of contraction of cycles along degenerations of real curves.\\

This paper is organised as follows. In section \ref{sec:set}, we introduce the obstructions and the main results more precisely. In section \ref{sec:toric}, we recall some basic facts about toric surfaces and how to detect the obstructions on a Newton polygon corresponding to the linear system. In section \ref{sec:harnack}, we generalize some results of \cite{KO} in order to produce our first vanishing cycles (see Theorem \ref{thm:acycle}). Moreover, we provide a trivialization of the universal curve over the space of simple Harnack curves (see Proposition \ref{prop:triv}). It allows us to describe the vanishing cycles in a combinatorial fashion on the Newton polygon (see Corollary \ref{cor:triv}). In section \ref{sec:tropi}, we use a 1-parameter version of Mikhalkin's approximation Theorem (see Theorem \ref{thm:rea}) to produce elements of $\im(\mu)$ given as some particular weighted graphs in the Newton polygon. In section \ref{sec:combi}, we play with this combinatorics to prepare the proof of Theorems \ref{thm:main1} and \ref{thm:main2} which we give in section \ref{sec:dis2}. At the end of the paper, we motivate Conjecture \ref{conj:general} with some extra constructions.

\begin{ack} The authors are grateful to Denis Auroux,  Sylvain Courte, Simon K. Donaldson, Christian Haase, Tobias Ekholm, Ilia Itenberg, Grigory Mikhalkin, and all the participants of the learning seminar on Lefschetz fibrations of Uppsala. The two authors were supported by the Knut and Alice Wallenberg Foundation.
\end{ack}

\setcounter{tocdepth}{2}
\tableofcontents

\section{Setting and statements}\label{sec:set}

\subsection{Setting}\label{setting}

Let $\Xd$ be a smooth complete toric surface. Take an ample line bundle $\Li$ on $\Xd$ and denote by $\LL = \mathbb{P}(H^0(\Xd,\Li))$ the complete linear system associated to it.

\begin{Definition}
  The discriminant of the linear system $\LL$ is the subset $\D\subset \LL$ consisting of all the singular curves in $\LL$.
\end{Definition}

Notice that since $\Xd$ is toric, $\Li$ is in fact very ample (see \cite[Corollary 2.15]{Oda}) and the embedding of $\Xd$ in the dual $\LL^*$ of the linear system identifies $\D$ with the dual variety of $\Xd$. In particular, let us recall the following classical facts.

\begin{Proposition}[6.5.1 in \cite{DIK}]
If $\Li$ is an ample line bundle on a smooth complete toric surface $\Xd$, then the discriminant $\D$ is an irreducible subvariety of $\LL$ of codimension at least $1$. It is of codimension $1$ if and only if there exists an element of $\D$ whose only singularity is a nodal point. In this case, the smooth set of $\D$ is exactly the set of such curves.
\end{Proposition}

From now on, we will assume that the discriminant $\D$ is a hypersurface in $\LL$. Let $C'\in \D$ be a curve whose only singularity is a node, which we denote by $p\in C'\subset \Xd$. Then, from \cite[Proposition 2.1]{ACG2} we deduce that there exist a neighborhood $U\subset \Xd$ of $p$, a neighborhood $B\subset \LL$ of $C'$ such that the restriction of the universal curve
\[
\{(x,C)\in U\times B\ |\ x\in C\} \rightarrow B
\]
is isomorphic to
\begin{equation}\label{local}
\{(z,w,t_1,\ldots,t_{\dim(\LL)})\in \C^2\times\C^{\dim(\LL)}\ |\ z^2 + w^2 = t_1\} \rightarrow \C^{\dim(\LL)}
\end{equation}
in a neighborhood of $(0,\ldots,0)$. The restriction of the previous map to the locus where $t_1\neq 0$ and $\frac{z^2}{t_1},\frac{w^2}{t_1} \geq 0$ is fibered in circles whose radius goes to $0$ as $t_1$ goes to $0$.

 Take now $C_0\in\LL$ a smooth curve and $\gamma : [0,1] \rightarrow \LL$ a smooth path such that $\gamma(0) = C_0$ and $\gamma^{-1}(\D) = \{1\}$, with $\gamma(1) = C'$. We can find $\epsilon>0$ such that $\gamma([1-\epsilon,1])$ is included in $B$. From Ehresmann theorem \cite[Theorem 9.3]{Voisin1}, the universal curve over $\LL\setminus \D$ is a smoothly locally trivial fibration. In particular, we can trivialize its pullback to $[0,1-\epsilon]$ via $\gamma$. Using this trivialization and the local form of the universal curve over $B$, the circles obtained in the local form \eqref{local} of the universal curve induce a smooth simple closed curve $\delta$ on $C_0$. The isotopy class of this curve does not depend on the choices made (see \cite[\S 10.9]{ACG2}).

\begin{Definition}
  We call the isotopy class of $\delta$ in $C_0$ a vanishing cycle of $C_0$ in the linear system $\LL$.
\end{Definition}

The present paper is motivated by a question of Donaldson (Question 1 in \cite{Do}), which we formulate as follows.

\begin{Question}\label{qvc}
Which isotopy classes of smooth simple closed curves on $C_0\in \LL$ are vanishing cycles in $\LL$?  
\end{Question}

Starting from a vanishing cycle $\delta$ on $C_0$, one can construct another one using the geometric monodromy map which is defined as follows. Taking a path $\gamma :[0,1]\rightarrow \LL\setminus \D$ from $C_0$ to itself, the pullback of the universal curve over $\gamma$ is trivialisable and a choice of trivialisation induces an orientation preserving diffeomorphism $\phi$ of $C_0$. In fact, the class of $\phi$ in the mapping class group $\MCG(C_0)$ of $C_0$ only depends on the homotopy class of $\gamma$. Thus, we obtain the geometric monodromy morphism
\[
\mu : \pi_1(\LL\setminus\D,C_0)\rightarrow \MCG(C_0).
\]

If $\delta$ is a vanishing cycle on $C_0$, then it follows from the definitions that for any element $\phi \in \im(\mu)$, $\phi(\delta)$ is also a vanishing cycle. Moreover, the irreducibility of $\D$ implies the following.

\begin{Lemma}[\cite{Voisin2}, Proposition 3.23]\label{irred}
If $\delta$ is a vanishing cycle of $C_0$ in $\LL$, then all the other vanishing cycles are obtained from $\delta$ under the action of the image of the geometric monodromy map. In particular, if $\delta$ and $\delta'$ are two vanishing cycles of $C_0$ in $\LL$, then the pairs $(C_0,\delta)$ and $(C_0,\delta')$ are homeomorphic.  
\end{Lemma}

As we will see later (see Theorem \ref{thm:acycle}), since we assume that $\LL$ is ample, as soon as the genus of $C_0$ is at least $1$ there always exists a non-separating vanishing cycle on $C_0$, $i.e.$ a vanishing cycle $\delta$ such that $C_0\setminus \delta$ is connected. It then follows from Lemma \ref{irred} that all the other vanishing cycles have the same property. Let us denote by $\I$ the set of isotopy classes of non-separating simple closed curves on $C_0$. Question \ref{qvc} now becomes: are all the elements of $\I$ vanishing cycles? 

By definition, if $\delta\in \I$ is a vanishing cycle, we can find a path $\gamma:[0,1]\rightarrow \LL$ starting at $C_0$ and ending on $\D$ defining $\delta$. Replacing the end of $\gamma$ by the boundary of a small disk transverse to $\D$ we obtain an element of $\pi_1(\LL\setminus\D,C_0)$. The computation of its image by $\mu$ gives the following.

\begin{Proposition}[\cite{ACG2}, \S 10.9]\label{dehnm}
If a simple closed curve $\delta\in \I$ is a vanishing cycle in $\LL$, then the Dehn twist $\tau_{\delta}$ along $\delta$ is in the image of $\mu$.  
\end{Proposition}

 We do not know if the converse to Proposition \ref{dehnm} is true in general. However, Lemma \ref{irred} shows that if the geometric monodromy is surjective then all the elements of $\I$ are vanishing cycles. Conversely, if all the elements of $\I$ are vanishing cycles, then the geometric monodromy is surjective since $\MCG(C_0)$ is generated by the Dehn twists around the elements of $\I$. This suggests the following analogue to Question \ref{qvc} which we will study in the rest of the paper.

\begin{Question}\label{qmo}
  When is the geometric monodromy map $\mu$ surjective? More generally, describe the subgroup $\im(\mu)$.
\end{Question}

\subsection{Roots of the canonical bundle and monodromy}\label{roots}

In some cases, the smooth curves in $\LL$ will inherit some geometric structure from the ambient toric surface. As a partial answer to Question \ref{qmo}, we will see below how this structure might prevent the map $\mu$ to be surjective. To make this more precise, we introduce subgroups of the mapping class group which will come into play. 

Let $\pi : \CC\rightarrow B$ be a holomorphic family of compact Riemann surfaces of genus $g>1$ and $n>0$ be an integer which divides $2g-2$. An $n\textsuperscript{th}$ root of the relative canonical bundle $K_{rel}(\CC)$ of $\CC$ is a holomorphic line bundle $\LS$ on $\CC$ such that $\LS^{\otimes n}=K_{rel}(\CC)$. Denote by $R_n(\CC)$ the set of $n\textsuperscript{th}$ roots of $K_{rel}(\CC)$. If the base $B$ is a point, then $R_n(\CC)$ is finite of order $n^{2g}$. 

Choose a base point $b_0\in B$, and let $C_0 = \pi^{-1}(b_0)$. On the one hand, there is a restriction map from $R_n(\CC)$ to $R_n(C_0)$ and on the other, there is a monodromy morphism $\mu_{\CC} : \pi_1(B,b_0)\rightarrow \MCG(C_0)$ defined as in \S \ref{setting}. Recall from \cite[Corollary p.73]{sipe} that the group $\MCG(C_0)$ acts on $R_n(C_0)$. Moreover, we have the following result.

\begin{Proposition}[\cite{earlesipe}, Theorem 2]\label{earsip}
  If $\LS\in R_n(\CC)$, then all the elements in $\im(\mu_{\CC})$ preserve $\LS_{|C_0}\in R_n(C_0)$.
\end{Proposition}

If $S\in R_n(C_0)$, denote by $\MCG(C_0,S)$ the subgroup of $\MCG(C_0)$ which preserves $S$. It is a subgroup of finite index which is proper in general (see \cite{sipe}).

Coming back to our problem, we obtain the following.

\begin{Proposition}\label{root}
Let $\LL$ be an ample complete linear system on $\Xd$ such that its generic element is a curve of genus at least $2$. Suppose that the adjoint line bundle $K_{\Xd} \otimes \Li$ admits a root of order $n\geq 2$, $i.e.$ a line bundle $\LS$ over $\Xd$ such that $\LS^{\otimes n} = K_{\Xd}\otimes \Li$. Then the restriction of $\LS$ to $C_0$ is a root of order $n$ of $K_{C_0}$ and the image of $\mu$ is a subgroup of $\MCG(C_0,\LS_{|C_0})$.
\end{Proposition}

\begin{proof}
The fact that the restriction of $\LS$ to $C_0$ is a root of order $n$ of $C_0$ comes from the adjunction formula (see \cite{GH})
\[
\left(K_{\Xd}\right)_{|C_0} = K_{C_0}\otimes \Li_{|C_0}^*.
\]
To prove the rest of the proposition, we show that this formula is true in family. 

To that end, choose a hyperplane $H$ in $\LL$ which does not contain $C_0$. Let 
\[
\CC = \{(x,C)\in \Xd\times\LL\setminus(\D\cup H)\ |\ x\in C\}
\]
be the universal curve with projection $\pi : \CC\rightarrow \LL\setminus(\D\cup H)$ and evaluation map $ev : \CC \rightarrow \Xd$. Then the pullback of $K_{\Xd}\otimes\Li$ to $\CC$ by $ev$ is isomorphic to the relative canonical bundle of $\CC$. Indeed, denote by $T_{rel}\CC$ the kernel of $\di \pi$. By definition, $K_{rel}(\CC)$ is the dual of $T_{rel}\CC$. Moreover, since all the curves in $\LL\setminus\D$ are smooth the differential of the evaluation map induces an injective morphism $T_{rel}\CC\rightarrow ev^*T\Xd$. On the other hand, since we removed $H$ from $\LL$, we can choose a section $s$ of the pullback of $\Li$ over $\Xd\times\LL\setminus(\D\cup H)$ which vanishes along $\CC$. Its derivative is well-defined along $\CC$ and induces a morphism $ev^*T\Xd\rightarrow ev^*\Li$ which is surjective since the curves in $\LL\setminus\D$ are transversally cut. It follows also that we have an exact sequence
\[
0\rightarrow T_{rel}\CC \rightarrow ev^*T\Xd\rightarrow ev^*\Li\rightarrow 0.
\]
Thus $ev^*K_{\Xd}\otimes ev^*\Li = K_{rel}(\CC)$. In particular, $ev^*\LS^{\otimes n} = K_{rel}(\CC)$, and $ev^*\LS$ is an element of $R_n(\CC)$. The result follows from Proposition \ref{earsip} and the fact that the morphism $\pi_1(\LL\setminus(\D\cup H),C_0)\rightarrow\pi_1(\LL\setminus\D,C_0)$ is surjective.
\end{proof}

Earle and Sipe showed in \cite[Corollary 5.3]{earlesipe} that for any Riemann surface of genus $g>2$, there exists an element of its mapping class group which does not preserve any $n\textsuperscript{th}$ root of its canonical bundle for all $n>2$ dividing $2g-2$. In particular, Proposition \ref{root} gives the following.

\begin{Corollary}\label{cor:obstruction}
  Let $\Xd$ be a smooth and complete toric surface and let $\Li$ be an ample line bundle on $\Xd$. Assume that the curves in $\LL$ are of  arithmetic genus at least $3$. Fix a smooth curve $C_0$ in $\LL$.

If $K_X\otimes \Li$ admits a $n\textsuperscript{th}$ root with $n>2$, then the geometric monodromy $\mu$ is not surjective. In particular, there exists a non-separating simple closed curve on $C_0$ which is not a vanishing cycle.\qed
\end{Corollary}

\begin{Remark}
  Since the Picard group of $\Xd$ is free \cite[Proposition p.63]{Ful}, the $n\textsuperscript{th}$ root of $K_{\Xd} \otimes \Li$ is unique when it exists.
\end{Remark}

\subsection{Main results}

In the theorems below, $\Xd$ is a smooth and complete toric surface and $\Li$ is an ample line bundle on $\Xd$. We assume that the curves in $\LL$ are of  arithmetic genus at least $1$. This implies that the adjoint line bundle $K_{\Xd}\otimes \Li$ of $\Li$ has an empty base locus (see Proposition \ref{prop:poly} below). Finally, we denote by $d$ the dimension of the image of the map $\Xd \rightarrow \LK^*$ and we fix a smooth curve $C_0$ in $\LL$.

\newpage

\begin{Theorem}\label{thm:main1}
The monodromy map $\mu$ is surjective if and only if one of the following is satisfied:
\begin{enumerate}
\item $d=0$,
\item $d=2$ and $K_{\Xd}\otimes \Li$ admits no root of order greater or equal to $2$.
\end{enumerate}
\end{Theorem}

The case $d=1$ corresponds to the hyperelliptic case and will be treated in the follow-up paper \cite{article2}.

If $n$ denotes the largest order of a root $\LS$ of $K_{\Xd}\otimes \Li$, there is a distinction between $n$ odd and $n$ even. When $n$ is even and $\LS$ is the $n\textsuperscript{th}$ root of $K_{\Xd}\otimes \Li$, the restriction of $\LS^{\otimes \frac{n}{2}}$ to $C_0$ is a $Spin$ structure on $C_0$. We study this case in another paper \cite{article2}. Let us simply say that when $n$ is even, the non-surjectivity of the monodromy already appears through the algebraic monodromy map

\[
\mmu : \pi_1(\LL\setminus \D,C_0)\rightarrow \spaut(H_1(C_0,\Z)),
\]
obtained by composing $\mu$ with the natural map $\MCG(C_0)\rightarrow \spaut(H_1(C_0,\Z))$, where $\spaut(H_1(C_0,\Z))$ is the group of automorphisms of $H_1(C_0,\Z)$ which preserve the intersection form.

When $n$ is odd, the obstruction described in Corollary \ref{cor:obstruction} is not detected by the algebraic monodromy.

\begin{Theorem}\label{thm:main2}
The algebraic monodromy map $\mmu$ is surjective if and only one of the following is satisfied:
\begin{enumerate}
\item $\mu$ is surjective,
\item $d=2$ and $K_{\Xd}\otimes \Li$ admits no root of order $2$.
\end{enumerate}
\end{Theorem}

In the general case, we conjecture the following.

\begin{Conjecture}\label{conj:general}
 Assume that $d=2$. Let $n$ be the largest order of a root $\LS$ of $K_{\Xd}\otimes \Li$. The image of $\mu$ is exactly $\MCG(C_0,\LS_{|C_0})$.
\end{Conjecture}

We motivate this conjecture in the present paper (see \S \ref{sec:conj}) and in a follow-up paper (\cite{article2}).

We give the proof of the above theorems in the sections \ref{sec:dless2} and \ref{sec:dis2} where we treat separately the cases $d=0$ and $d=2$.

\begin{Remark}
  \begin{itemize}
  \item  The amplitude and genus requirements in the Theorems above guarantee that the discriminant $\D$ is of codimension $1$ in $\LL$ and that the vanishing cycles are non-separating (see Theorem \ref{thm:acycle}).
  \item  Question \ref{qvc} can also be stated in a symplectic setting. That is, if one allows the almost-complex structure on $\Xd$ to vary along a degeneration of $C_0$, can we obtain more vanishing cycles than in the algebraic case? Does the obstruction coming from the root of $K_{\Xd}\otimes \Li$ survive in this setting?
  \end{itemize}
\end{Remark}

To prove Theorems \ref{thm:main1} and \ref{thm:main2}  we use techniques from tropical geometry to construct explicit one-parameter families of curves in $\LL\setminus\D$ producing some simple elements of the mapping class group using Mikhalkin's realisability theorem (Theorem \ref{thm:rea}). In the non-$Spin$ case, we manage to construct a family of elements of $\MCG(C_0)$ whose image in $\spaut(H_1(C_0,\Z))$ coincides with the one of Humphries's generating family of the mapping class group (see \S\S \ref{sec:prime} and \ref{sec:odd}).

\section{Toric surfaces, line bundles and polygons}\label{sec:toric}

We recall some facts about toric surfaces (for more details, see e.g \cite{Ful} or \cite{GKZ}). Let $\Xd$ be a smooth complete toric surface associated to the fan $\Sigma\subset \Z^2\otimes\R$. Denote by $\Sigma(1)$ the set of $1$-dimensional cones in $\Sigma$. For each element $\epsilon\in\Sigma(1)$ let $u_\epsilon$ be the primitive integer vector in $\epsilon$ and $D_{\epsilon}\subset\Xd$ the associated toric divisor.

Take a line bundle $\Li=\OO_{\Xd}(\sum_{\epsilon\in\Sigma(1)}a_{\epsilon}D_{\epsilon})$ on $\Xd$ and define
\[
\Delta_{\Li} = \{v\in\R^2\ |\ \forall\epsilon\in\Sigma(1),\ \langle v, u_\epsilon \rangle \geq -a_{\epsilon}\}.
\]
Notice that $\Delta_{\Li}$ is in fact only well-defined up to translation by an integer vector.

If $\Li$ is nef, $\Delta_{\Li}$ is a convex lattice polygon (in this paper, lattice polygons are bounded by convention) and $\Sigma$ is a refinement of its normal fan. More precisely, $\Delta_{\Li}$ has one edge for each $\epsilon\in\Sigma(1)$ and the integer length $l_{\epsilon}$ of this edge is equal to the intersection product $(\sum_{\epsilon'\in\Sigma(1)}a_{\epsilon'}D_{\epsilon'})\bullet D_{\epsilon}$ which can be $0$. In particular, we have the following Proposition.

\begin{Proposition}\label{prop:rootandlength}
  Let $\Li$ be a nef line bundle on a smooth complete toric surface $X$ with fan $\Sigma$. Let $n$ be the largest order of a root of $\Li$. Then
\[
 n=\gcd\left(\left\{ l_{\epsilon},\epsilon\in\Sigma(1)\right\}\right).
\]
\end{Proposition}

\begin{proof}
 Let $m = \gcd\left(\left\{ l_{\epsilon},\epsilon\in\Sigma(1)\right\}\right)$ and let $\LS $ be a root of order $n$ of $\Li$. Then for any $\epsilon\in\Sigma(1)$,
\[
l_{\epsilon} = \Li\bullet D_{\epsilon} = n(\LS \bullet D_{\epsilon}).
\]
Thus $n$ divides $m$.

On the other hand, let $U$ be the matrix whose rows are given by the vectors $u_\epsilon$. Up to translation, we can assume that $0$ is a vertex of $\Delta_{\Li}$.  In particular, $\frac{1}{m}\Delta_{\Li}$ is  still a convex lattice polygon. Moreover, if $\Delta_{\Li}$ is given by the system $Uv\geq -a$, then $\frac{1}{m}\Delta_{\Li}$ is given by $Uv\geq -\frac{a}{m}$. In particular, $m$ divides all the coordinates of $a$. Thus, we can write
\[
\Li = \sum_{\epsilon\in\Sigma(1)}a_{\epsilon}D_{\epsilon} = m \sum_{\epsilon\in\Sigma(1)}\frac{a_{\epsilon}}{m}D_{\epsilon},
\]
and $\Li$ has a root of order $m$. Thus $m$ divides $n$.
\end{proof}

\begin{Definition}
A convex lattice polygon $\Delta$ is even (respectively odd, respectively prime) if 
$
\gcd\left(\left\{ l_{\epsilon},\epsilon \text{ edge of } \Delta \right\}\right)
$
is even (respectively odd, respectively prime).\\
A nef line bundle $\Li$ on a smooth complete toric surface $X$ is even (respectively odd, respectively prime) if the polygon $\Delta_{\Li}$ is even (respectively odd, respectively prime).
\end{Definition}

If $\Li$ is ample, then the normal fan of $\Delta_{\Li}$ is $\Sigma$. Moreover, since $\Xd$ is smooth, the polygon $\Delta_{\Li}$ is also smooth, $i.e.$ any pair of primitive integer vectors directing two consecutive edges of $\Delta_{\Li}$ generates the lattice. Conversely, a smooth convex lattice polygon $\Delta$ of dimension 2 defines a smooth complete toric surface $\Xd$ as follows: label the elements of $\Delta \cap \Z^2 = \{ (a_1,b_1),...(a_m,b_m)\}$ and consider the monomial embedding $(\C^\ast)^2 \rightarrow \cp{m}$ given by
\[ (z,w) \mapsto \left[ z^{a_1}w^{b_1},..., z^{a_m}w^{b_m}\right]. \]
Then $\Xd$ is defined as the closure of $(\C^\ast)^2$ in $\cp{m}$. In particular, it always comes with $(z,w)$-coordinates and the associated complex conjugation. The line bundle $\Li$ on $\Xd$  given by the inclusion in $\cp{m}$ is such that $\Delta = \Delta_{\Li}$. If $\Delta'$ is obtained from $\Delta$ by an invertible affine transformation $A:\R^2\rightarrow\R^2$ preserving the lattice, then $A$ induces an isomorphism between the two toric surfaces obtained from $\Delta$ and $\Delta'$ which pulls back $\Li'$ to $\Li$. Indeed, the lattice $\Z^2\subset \R^2$ is naturally isomorphic to the space of characters on $\ttor$ via the $(z,w)$-coordinates. The map $A^\vee : \ttor \rightarrow \ttor$ dual to $A$ induces the desired isomorphism of toric surfaces.

When $\itr(\Delta)\cap \Z^2$ is non-empty, we can consider the adjoint polygon $\da$ of $\Delta$ defined as the convex hull of the interior lattice points of $\Delta$. The number of lattice points of $\da$ is equal to the arithmetic genus of the curves in $\LL$ which we denote by $g_{\Li}$, see \cite{Kho}. Denote also by $b_\Li$ the number of lattice points in $\partial \Delta$. $b_\Li$ is then equal to the intersection multiplicity of a curves in $\LL$ with the divisor $(\sum_{\epsilon\in\Sigma(1)}D_{\epsilon})$.

\begin{Proposition}\label{prop:poly}
  Let $\Xd$ be a smooth complete toric surface with fan $\Sigma$ and let $\Li$ be an ample line bundle on $\Xd$ such that $g_{\Li}\geq 1$.
  \begin{enumerate}
  \item The adjoint line bundle $K_{\Xd}\otimes\Li$ is nef with empty base locus and $\Delta_{K_{\Xd}\otimes\Li} = \da$. In particular, for each $\epsilon\in\Sigma(1)$ the associated edge of $\da$ has length $l_{\epsilon}-D_{\epsilon}^2-2$. (\cite[Lemma 2.3.1 and Proposition 2.4.2]{Koe}, \cite[Theorem 2.7]{Oda})
\item Let $\phi_{K_{\Xd}\otimes \Li} : \Xd\rightarrow \LK^*$ be the natural map. We have $\dim(\im(\phi_{K_{\Xd}\otimes \Li})) = \dim(\da)$.
\item If $\dim(\da) = 2$, then $\da$ is smooth. (\cite[Lemma 5]{Ogata})
  \end{enumerate}\qed
\end{Proposition}

If $\Delta$ is a smooth convex lattice polygon, we call vertices its extremal points.

\section{Simple Harnack curves}\label{sec:harnack}

Let $\Delta \subset \Z^2 \otimes \R $ be a smooth 2-dimensional convex lattice polygon, $\Xd$ the associated toric surface and $\Li$ the line bundle on $\Xd$ given by $\Delta$. Recall that $\Xd$ comes with an open dense torus $(\C^\ast)^2 \subset \Xd$ of coordinates $(z,w)$. These coordinates induce a complex conjugation on both $\Xd$ and $\Li$. Define the amoeba map $\A : (\C^\ast)^2 \rightarrow \R^2$ by 
\[\A(z,w) = (\log \vert z \vert, \log \vert w \vert ). \]
For any algebraic curve $C \subset \Xd$, denote $C^\circ := C \cap (\C^\ast)^2$ and $\R C^\circ := C \cap (\R^\ast)^2$.

\begin{Definition} A non degenerate real section $C \subset \Xd$ of $\Li$ is a (possibly singular) simple Harnack curve if the amoeba map $ \A : C^\circ \rightarrow \R^2$ is at most 2-to-1.
\end{Definition}

\begin{Remark} The original definition of \sh curves can be found in \cite{MR}, as well as its equivalence with the definition given above.
\end{Remark}

For a \sh curve $C$, the amoeba map $\A$ is in fact 1-to-1 on $\R C^\circ $ and the outer boundary of $\A(C^\circ)$ is the image of a single connected component of $\R C$. Hence, $\A(C^\circ)$ determines $C$ up to sign. We get rid of this sign ambiguity as follows. Choose a vertex $v \in \Delta$ with adjacent edges $\epsilon_1$ and $\epsilon_2$. By the original definition \cite{Mikh}, there is a unique connected arc $\alpha \subset \R C^\circ$ joining $\epsilon_1$ to $\epsilon_2$. We restrict ourselves to the set of smooth \sh curves for which $\alpha$ sits in the $(+,+)$-quadrant of $\R^2$. For technical reasons, we also assume that $C$ intersects the divisor $X \setminus \ttor$ transversally, implying that $C^\circ$ is a compact Riemann surface with $b_\Li$ points removed. We denote by $\Hd \subset \PSec$ the set of such \sh curves. 

By \cite{Mikh}, any curve $C \in \Hd$ is such that $b_0(\R C)= g_\Li +1$, $i.e.$ $C$ is maximal, and exactly $g_\Li$ connected components of $\R C$ are contained in $(\C^\ast)^2$. We refer to them as the $A$-cycles of $C$.

\begin{Theorem}\label{thm:acycle}
Any  $A$-cycle of a curve $C_0 \in \Hd$ is a vanishing cycle in the linear system $\vert \Li \vert$.
\end{Theorem}

\begin{Theorem}\label{thm:connected}
The set $\Hd$ is connected.
\end{Theorem}

The latter theorems are given in \cite{KO} in the case $\Xd = \cp{2}$. The proofs we provide below rely mainly on the latter reference. We only provide the extra arguments that we require. Before doing so, we recall one of the main ingredients. We identify $\Sec$ with the space of Laurent polynomial with Newton polygon $\Delta$. For any $f \in \Sec$, the Ronkin function $N_f : \R^2 \rightarrow \R$ given by 
\[ N_f (x,y) = \frac{1}{(2i\pi)^2} \iint_{\substack{\vert z \vert = e^{x} \\ \vert w \vert = e^{y}  }} \log \vert f(z,w) \vert \frac{dz dw}{zw} \]
is convex, piecewise linear with integer slope on every connected component of $\R^2 \setminus \A(C^\circ)$, where $C= f^{-1}(0) \subset \Xd$. It induces the order map 
\[  
\begin{array}{rcl}
\pi_0(\R^2 \setminus \A(C^\circ)) & \rightarrow & \Delta \cap \Z^2\\
U & \mapsto & grad \, N_f (U)
\end{array}.
\]
For $C \in \Hd$, the latter map is a bijection mapping any compact connected component of $\R^2 \setminus \A(C^\circ)$ in $\da \cap \Z^2$.
Note that for two polynomials defining the same fixed curve $C$, the associated Ronkin functions differ only by an additive constant. It follows that the map $(grad \, N_f) \circ \A : C^\circ \rightarrow \R^2$ does not depend on the choice of $f$. We refer to \cite{FPT} and \cite{PR} for more details. 

The proof of Theorem \ref{thm:acycle} relies on the fact that \cite[Proposition 6]{KO} extends to the present case. To show that, we need an explicit description of the space of holomorphic differentials on a curve $C\in \Hd$.

\begin{Lemma}\label{lem:difform}
Let $C\in\Hd$ and $f\in H^0(\Xd,\Li_\Delta)$ such that $C=\overline{\left\lbrace f=0\right\rbrace}$. The space of holomorphic differentials on $C$ is isomorphic to the space of sections of $\Li_{\Delta_a}$ via the map
\[ h(z,w) \mapsto \dfrac{h(z,w)}{\partial_w f(z,w)\, zw}dz. \]
\end{Lemma}

\begin{proof}
The space of sections of $\Li_{\Delta_a}$ has the expected dimension $g_\Li$ and two such meromorphic differentials are linearly independant for different $h$'s. It remains to show that they are in fact holomorphic. We proceed in two steps.

Assume first that $\Delta \subset \N^2$ has a vertex at the origin with adjacent edges given by the coordinates axes. Then, the plane $\C^2$ given by the coordinates $(z,w)$ provides a chart of $\Xd_\Delta$ at this vertex In such case, $h(z,w)/zw$ is a polynomial and the poles of $1/\partial_w f(z,w)$ are compensated by the vanishing of $dz$ on $\C^2\cap C$. The latter fact follows from a straightforward computation, using a local parametrization of $C$. It implies that any differential
\[ \varOmega = \dfrac{h(z,w)}{\partial_w f(z,w)\, zw}dz \]
is holomorphic on $\C^2\cap C$. 

We now show that one can apply the following argument at any vertex of $\Delta$, which concludes the proof. It relies on the following claim: for any lattice preserving transformation $A : \R^2 \rightarrow \R^2$ sending $\Delta$ to some $\Delta'$ and dual map $A^\vee:\ttor \rightarrow\ttor$, then 
\[ A^\vee_\ast \varOmega = \dfrac{h'(s,t)}{\partial_t f'(s,t)\, st}ds \]
where $h':=A_\ast h \in \Li_{\Delta'_a}$, $f':=A_\ast f \in \Li_{\Delta'}$ and $(s,t):=A^\vee(z,w)$.  Any such map $A$ can be decomposed into an integer translation and an $Sl_2(\Z)$-linear map. If $A$ is a translation by some vector $(a,b)\in\Z^2$,  $A^\vee$ is the identity. One has $h'=z^aw^bh$, $f'=z^aw^bf$. On the curve $C$, one has $\partial_w f'(z,w) =z^aw^b\partial_w f(z,w)$ as $f_{\vert C}=0$. It follows that
\[ \dfrac{h'(z,w)}{\partial_w f'(z,w)\, zw}dz  = \varOmega.\]

If $A(n,m)=(dn-bm, \, -cn +am)$, then $A^\vee(z,w)= (z^aw^b,z^cw^d)$, and by definition $h(z,w)=h'(s,t)$. Set the logarithmic coordinates $(\mathbf{z}, \mathbf{w})=(\log(z),\log(w))$ and $(\mathbf{s}, \mathbf{t})=(\log(s),\log(t))$, and define $\mathbf{f} (\mathbf{z}, \mathbf{w}) := f(z,w)$ and $\mathbf{f}' (\mathbf{z}, \mathbf{w}) := f'(z,w)$. Observing that $ \mathbf{f}'(\mathbf{s},\mathbf{t}) = \mathbf{f}(d\mathbf{s}-b \mathbf{t}, -c\mathbf{s}+a \mathbf{t}) $ and that 
$$ \dfrac{d\mathbf{z}}{\partial_\mathbf{w} \mathbf{f}(\mathbf{z},\mathbf{w})} = - \dfrac{d\mathbf{w}}{\partial_\mathbf{z} \mathbf{f}(\mathbf{z},\mathbf{w})} $$
on $\left\lbrace \mathbf{f}=0 \right\rbrace$, we compute 
\begin{flushleft}
$\qquad \displaystyle \dfrac{ds}{\partial_t f'(s,t)\, st}=
\dfrac{d\mathbf{s}}{\partial_\mathbf{t} \mathbf{f}'(\mathbf{s},\mathbf{t})}=
 \dfrac{a\, d\mathbf{z} + b\, d\mathbf{w}}{-b\, \partial_\mathbf{z} \mathbf{f}(\mathbf{z},\mathbf{w}) + a\, \partial_\mathbf{w} \mathbf{f}(\mathbf{z},\mathbf{w}) }$
\end{flushleft}
 \begin{flushright}
$  = \dfrac{d\mathbf{z}}{\partial_\mathbf{w} \mathbf{f}(\mathbf{z},\mathbf{w})} =
\dfrac{dz}{\partial_w f(z,w)\, zw}.  \qquad$
\end{flushright}
\end{proof}

\begin{proof}[Proof of Theorem \ref{thm:acycle}]
To begin with, we briefly describe the settings of \cite[Propositions 6 and 10]{KO}. Consider the space $H$ of simple Harnack curves in $\LL$ with fixed intersection points with $\Xd \setminus (\C^\ast)^2$. Note that $H$ is closed and has dimension $g_\Li$, and the only singular curves in $H$ have isolated real double points (see \cite{MR}).

Consider then the continuous and proper map $Area : H \rightarrow \R_{\geq 0}^{g_\Li}$ that associates to any curve $C$ the area of the holes of $\A(C^\circ)$. We show below that the map $Area$ is a local diffeomorphism. In particular, it is a covering map. Now for any curve $C \in H$, denote $Area(C)=(a_1, \, a_2, \, ..., \, a_{g_\Li})$ and $T_j :=\left\lbrace (a_1, \, ..., \, t a_j, \, ..., \, a_{g_\Li}) \, \vert \, t\in [0,1] \right\rbrace$, for $1\leq j \leq g_\Li$. As the map $Area$ is a covering map, we can lift the segment $T_j$ to a path in $\LL$ starting at $C$. By construction, this path ends at a curve whose only singularity is a node and the corresponding vanishing cycle is the $j$-th $A$-cycle of $C$.

Let us now show that $Area$ is a local diffeomorphism. To that aim, we show that \cite[Proposition 6]{KO} holds true in the present case. Namely, fixing a simple Harnack curve $C\in H$ defined by a polynomial $f$, we check that the intercepts of the $g_\Li$ affine linear functions supporting 
$N_f$ on the compact connected component of $\R^2 \setminus \A(C^\circ)$ provide local coordinates on $H$ near $C$. We reproduce the 
computation of \cite{KO}: from Lemma \ref{lem:difform}, we know that for the a defining polynomial $f \in H^0(\Xd,\Li)$ of $C$, any 
holomorphic differential on $C$ has the form
\[ \varOmega = \dfrac{h(z,w)}{\partial_w f(z,w)\, zw}dz \]
where $h\in H^0(\Xd,\Li_{\Delta_a})$. Taking $f$ to be a real polynomial, the tangent space of $H$ at $C$ is identified with the space of real polynomials $h$ in  $ H^0(\Xd,\Li_{\Delta_a})$. For any points $(x,y)$ contained in a compact component of $\R^2 \setminus \A(C^\circ)$, we can compute the variation of $N_f$ as in \cite[Proposition 6]{KO}
\begin{flushleft}
$  \qquad \displaystyle \left. \dfrac{d}{dt} N_{f+th} (x,y) \right|_{t=0}= \mathfrak{Re} \frac{1}{(2i\pi)^2} \iint_{\substack{\vert z \vert = e^{x} \\ \vert w \vert = e^{y}  }} \dfrac{h(z,w)}{f(z,w)} \frac{dz dw}{zw} $
\end{flushleft}
\begin{flushright}
$  \displaystyle = \mathfrak{Re} \frac{1}{2i\pi} \int_{\vert z \vert = e^{x}} \sum_{\substack{f(z,w)=0 \\ \vert w \vert < e^{y}  }} \dfrac{h(z,w)}{\partial_w f(z,w)} \frac{dz}{zw} =  \frac{1}{2i\pi} \int_{\beta} \dfrac{h(z,w)}{\partial_w f(z,w)} \frac{dz}{zw} \qquad $
\end{flushright}
where $\beta$ is the $B$-cycle $\left\lbrace \vert z \vert = e^{x}, \, \vert w \vert < e^{y}   \right\rbrace \cap C$. The latter is anti-invariant by complex conjugation, justifying that we can omit $\mathfrak{Re}$ in the last equality. Hence, the differential of the map giving the intercepts in terms of the coefficients of $f$ is the period matrix of $C$, which is invertible. 

Now the variational principle of Proposition 8 and the discussion opening Section 4.5 in \cite{KO} hold true in the present case as no argument depends on $\Delta$. It implies that the map giving the areas of the holes of $\A(C^\circ)$ in terms of the intercepts is a local diffeomorphism. We conclude that the map $Area$ is a local diffeomorphism.
\end{proof}

\begin{proof}[Proof of Theorem \ref{thm:connected}]
The strategy of the proof is as follows: we first show that any two curves $C_0$, $C_1 \in \Hd$ can be continuously degenerated to rational simple Harnack curves $C_0^\ast$ and $C_1^\ast$. Then we show that the space of rational simple Harnack curves is path connected, and construct a path inside the latter space joining $C_0^\ast$ to $C_1^\ast$. Finally, we show that the resulting path joining $C_0$ to $C_1$ can be deformed in $\Hd$.

First, we use the same argument as in the proof of Theorem \ref{thm:acycle} involving \cite[Proposition 10]{KO}. If $(a_{j,1}, \, a_{j,2}, \, ..., \, a_{j,g_\Li})$ is the vector of the area of the holes of $\A(C^\circ_j)$, $j=0,1$, we  can construct two continuous paths of simple Harnack curves $\left\lbrace C_{j,s} \, \vert \, s \in [0,1] \right\rbrace$ such that $C_{j,1}=C_j$ and the area of the holes of $\A(C^\circ_{j,s})$ is $s (a_{j,1}, \, ..., \, a_{j,g_\Li})$. Notice that $C_{j,s}\in \Hd$ for $s>0$. The curves $C_{j,0}$ are the curves $C_j^\ast$ announced above.

The space of rational Harnack curves admits a parametrization similar to the one given in \cite[Section 4.1]{KO}. Label the edges of $\Delta$ counterclockwise by $\epsilon_1$, ..., $\epsilon_n$, with respective primitive integer normal vectors $v_1=(a_1,b_1)$, ..., $v_n=(a_n,b_n)$ pointing outwards, and finally $l_j:=l_{\epsilon_j}$. Real rational curves of Newton polygon $\Delta$ can always be parametrized by
\begin{eqnarray}\label{eq:param}
 t \in \cp{1} \mapsto \Big( \alpha \prod_{ j=1}^n \prod_{k=1}^{l_j} (t-c_{jk})^{a_j}, \, \beta \prod_{ j=1}^n \prod_{k=1}^{l_j} (t-c_{jk})^{b_j}\Big) \in \Xd
\end{eqnarray}
with $\alpha$, $\beta \in \R^\ast$ and the $c_{jk}$ are either real or appear in complex conjugated pairs. Such representation is unique up to the action of $PGL_2(\R)$ on the parameter $t$. The same arguments as in \cite[Proposition 4]{KO} apply so that the curve \eqref{eq:param} is a rational simple Harnack curve if and only if all the $c_{jk}$ are real and
\begin{eqnarray}\label{eq:param2}
 c_{11}\leq ... \leq c_{1l_1} < c_{21} \leq ...\leq c_{2l_2} < ... < c_{n1}\leq ... \leq c_{nl_n}.   
\end{eqnarray}
For a curve in $\partial \Hd$, the signs of $\alpha$ and $\beta$ are determined by the sign convention made on $\Hd$.
The space of parameters $\alpha$, $\beta$ and $c_{jk}$'s satisfying \eqref{eq:param2} is connected. It implies that the space of rational simple Harnack curves in $\partial \Hd$ is connected. Hence, one can construct a continuous path $\left\lbrace C^\ast_t \, \vert \, t\in[0,1] \right\rbrace$ in this space joining $C_0^\ast$ to $C_1^\ast$.

The deformations of $C^\ast_t$ towards $\Hd$ with fixed points on $\Xd \setminus (\C^\ast)^2$ are given by real polynomials  $h\in H^0(\Xd,\Li_{\Delta_a})$ having prescribed signs at the nodes of $C^\ast_t$, once we choose a continuous path of defining real polynomials $f_t$ for $C^\ast_t$. Indeed, $h$ and $f_t$ need to have opposite sign around each node for an oval to appear. This sign distribution does not depend on $t$ as the topological pair $\big( (\R^\ast)^2, \R (C_t^\ast)^\circ\big)$ does not depend on $t$ either. The vanishing of the polynomials of $\R H^0(\Xd,\Li_{\Delta_a})$ at each of the node of $C^\ast_t$ impose independent conditions, cutting $\R H^0(\Xd,\Li_{\Delta_a})$ into $2^{g_\Li}$ orthant. Only the orthant corresponding to the sign prescription above leads to deformation into $\Hd$. It induces a continuous family of orthant in $t$ inside of which $\left\lbrace C^\ast_t \right\rbrace_t$ can be continuously deformed to a path in $\Hd$ with endpoints on $\left\lbrace C_{0,s} \right\rbrace_s$ and $\left\lbrace C_{1,s} \right\rbrace_s$. The result follows.
\end{proof}

\begin{Remark} We believe that the space of Harnack curves for any smooth toric surface admits a parametrization similar to the one given in \cite{KO} for non degenerate projective curves. In particular, we expect this set to be contractible.
\end{Remark}

We now show that $\mathbf{H}_\Delta:=\left\lbrace \big((z,w),C\big) \in \ttor \times \Hd, \, \vert \,  (z,w)\in C \right\rbrace$ is a trivial fibration over $\Hd$. For this purpose, we will again make use of the Ronkin function. It is shown in \cite{PR} that for $C\in\LL$ and $f$ a defining polynomial the map $(grad \, N_f) \circ \A : C^\circ \rightarrow \R^2$ takes values in  $\Delta$. In the Harnack case, one has the following stronger statement.

\begin{Lemma}\label{lem:difdelta} 
For a curve $C \in \Hd$ with a defining polynomial $f$, the map $(grad \, N_f) :  \itr\big(\A (C^\circ)\big)  \rightarrow \Delta$ is a diffeomorphism onto $ \itr(\Delta \setminus \Z^2)$.
\end{Lemma}

\begin{proof} This lemma follows from \cite[Lemma 4.4]{Kri}. We choose to give an alternative proof here. 

First, it is shown in \cite[Theorem 7]{PR} that the Hessian of $N_f$ is symmetric positive definite on $\itr\big(\A (C^\circ)\big)$, implying both that $grad \, N_f$ is a local diffeomorphism and that $N_f$ is strictly convex on $\itr\big(\A (C^\circ)\big)$. Let us show that $grad \, N_f$ is injective. Suppose there are two points $p, \, q \in \R^2$ such that $grad \, N_f \, (p) = grad \, N_f \, (q)$. By the convexity of $N_f$, it implies that $N_f$ is affine linear on the segment joining $p$ to $q$. By the above strict convexity, both $p$ and $q$ sit outside $\itr\big(\A (C^\circ)\big)$, which proves injectivity.

Now, it is shown in \cite[Theorem 4]{PR} that $\overline{\im(grad \, N_f)}= \Delta$. Suppose there is a point $p\in \itr(\Delta \setminus \Z^2)$ not in $\im(grad \, N_f)$ and consider a sequence $\left\lbrace p_n\right\rbrace_{n\in\N} \subset \im(grad \, N_f)$ converging to $p$. All the limit points of $\left\lbrace (grad \, N_f)^{-1}(p_n)\right\rbrace_{n\in\N}$ are either in $\partial \A (C^\circ)$ or escapes to infinity along  the tentacles of $\A (C^\circ)$. It implies that $p_n$ converges either to $ \Delta \cap \Z^2$ or to $\partial \Delta$. This leads to a contradiction. The surjectivity follows.
\end{proof}

We now show that this map provides a trivialisation the fibration $\mathbf{H}_\Delta \rightarrow \Hd$ via the following construction. First, consider the real oriented blow-up $\tilde{\Delta}$ of $\itr( \Delta ) \cup (\Delta \cap \Z^2)$ at its integer points. This operation replaces the neighbourhood of any integer point $p$ by $\left\lbrace x \in \R^2 \vert \,  \Vert x \Vert \geq 1 \right\rbrace$ or $\left\lbrace x \in \R\times \R_{>0} \vert \, \Vert x \Vert \geq 1 \right\rbrace$ depending whether $p$ sits in $\itr( \Delta)$ or not. Now, consider a copy $\tilde{\Delta}'$ of $\tilde{\Delta}$ given with opposite orientation and define
\[ C^\circ_\Delta := \tilde{\Delta}' \cup \tilde{\Delta} / \sim \]
where $\sim$ is the identification of the boundary of $\tilde{\Delta}'$ to the one of $\tilde{\Delta}$. It follows that $ C^\circ_\Delta$ is an oriented surface of genus $g_\Li$ with $b_\Li$ points removed. Denote by $ pr : C^\circ_\Delta \rightarrow \tilde{\Delta}$ the 2-to-1 projection and by $ conj :  C^\circ_\Delta \rightarrow C^\circ_\Delta$ the associated Deck transformation. By Lemma \ref{lem:difdelta}, we can define an  orientation preserving and equivariant map $R_C : (C \setminus \R C)^\circ \rightarrow C^\circ_\Delta$ such that $ pr \circ R_C = (grad \, N_f) \circ \A $. The map $R_C$ is unique.

\begin{Proposition}\label{prop:triv} 
The map $R_C$ extends to a diffeomorphism $R_C : C^\circ \rightarrow C^\circ_\Delta$. Moreover, the map $R:= (R_C,\, \id) : \mathbf{H}_\Delta \rightarrow C^\circ_\Delta \times \Hd$ is a trivialisation of the fibration $\mathbf{H}_\Delta \rightarrow \Hd$.
\end{Proposition}

Before getting into the proof, recall that the logarithmic Gauss map $\gamma : C^\circ \rightarrow \cp{1}$ is is the composition of any branch of the complex logarithm with the standard Gauss map. In coordinates $(z,w) \in (\C^\ast)^2$, 
\[ \gamma(z,w)= \left[z \partial_z f(z,w):  w \partial_w f(z,w)\right] \]
where $f$ is a defining polynomial for $C$.

\begin{proof}
For the first part of the statement, we show that $R_C$ extends to a diffeomorphism on one half of the curve $C^\circ$ and use the equivariance to conclude. 

Consider the affine chart $\left[u:v\right]\rightarrow u/v$ of $\mathbb{CP}^1$ and define the half $C^+$ of $C^\circ$ to be the closure of $\left\lbrace q \in C^\circ \, \vert \,  \mathfrak{Im}  \big( \gamma(q) \big) \geq 0 \right\rbrace$. Notice that it induces an orientation on $\R C^\circ$ and a lift $\gamma^+ : \R C^\circ \rightarrow S^1$ of $\gamma$. By the computation of the Hessian of $N_f$ in \cite[Equation (19)]{PR}, we deduce that the restriction to $C^+$ of $(grad \, N_f) \circ \A$ is a lift of the argument map $Arg : C^\circ \rightarrow (S^1)^2$ to its universal covering $\R^2$ after a rotation of $\pi/2$. Now let $p\in\R C^\circ$. Up to permutation of the coordinates and sign change, we can assume that $\gamma(p)=\left[\alpha:1\right]$ and that $p$ sits in the positive quadrant. We prove below that the restriction of $Arg$ to a neighbourhood of $p$ in $C^+$ lifts to the real oriented blow up of $(S^1)^2$ at $(0,0)$.

By \cite[Corollary 6]{Mikh}, the logarithmic Gauss map $\gamma$ has no critical point on $\R C^\circ$. Then $\gamma$ is 1-to-1 around $p$ and can be used as a local coordinate for $C$ around $p$. The local parametrization $\Log \circ \gamma^{-1} : t:=u/v \mapsto (\mathbf{z}(t),\mathbf{w}(t))$ of $\Log(C)$ around $\Log(p)$ satisfies
\[ - \mathbf{w}'(t) / \mathbf{z}'(t) = t \]
where the left-hand side is the map $\gamma$ seen via the affine chart of $\mathbb{CP}^1$ chosen above. It follows that 
\[  \mathbf{z} (t) = a_\mathbf{z} + \sum_{j\geq 1} a_j t^j  \; \text{ and } \; \mathbf{w} (t) = a_\mathbf{w} -\sum_{j\geq 1} \frac{j}{j+1}a_j t^{j+1} \]
where $a_\mathbf{z}, \, a_\mathbf{w}, \, a_j\in \R$. Recall that $Arg$ is the projection on $i\R^2$ in logarithmic coordinates. Hence, in the coordinate $t$, one has
\[ Arg(t) = \big(\mathfrak{Im} (\mathbf{z} (t)), \mathfrak{Im} (\mathbf{w} (t))\big) = \Big( \sum_{j\geq 1} a_j \mathfrak{Im}(t^j) ,  - \sum_{j\geq 1} \frac{j}{j+1}a_j \mathfrak{Im}(t^{j+1}) \Big).   \]
Now, let us denote $t=\alpha +i \beta$. A direct computation shows that 
\[ \lim_{\beta \rightarrow 0} \alpha \frac{\mathfrak{Im}(t^j)}{\beta} = \lim_{\beta \rightarrow 0} \frac{j}{j+1} \frac{\mathfrak{Im}(t^{j+1})}{\beta} = j \alpha^j \] 
providing the following explicit formula 
\[ \lim_{\beta \rightarrow 0^+} \dfrac{Arg(t)}{\Vert Arg(t) \Vert} = \dfrac{\mathbf{z}'(\alpha)}{\vert \mathbf{z}'(\alpha) \vert} \dfrac{(1,-\alpha)}{\Vert (1,-\alpha) \Vert}\]
for the extension of $Arg$ to the real oriented blow-up at $(0,0)$. The same computation shows that $\partial_\beta Arg (\alpha )= \mathbf{z}'(\alpha) (1,-\alpha)$.

Applying a rotation of $\pi/2$, we deduce that $R_C$ extends to $C^+$ and that its restriction to $\R C^\circ$ is given by $\gamma^+$. Together with Lemma \ref{lem:difdelta}, It implies that $R_C$ is differentiable and then induces a diffeomorphism from $C^+$ to $\tilde{\Delta}$.
From the equality $\partial_\beta Arg (\alpha )= \mathbf{z}'(\alpha) (1,-\alpha)$, we deduce moreover that the partial derivative $\partial_\beta R_C$ at any point of $\R C^\circ$ is normal to $\partial \tilde{\Delta}$. Hence, it extends to an equivariant diffeomorphism  on the whole $C^\circ$.

From the computation led above and lemma \ref{lem:difdelta}, we deduces that the map $R:= (R_C,\, \id) : \mathbf{H}_\Delta \rightarrow C^\circ_\Delta \times \Hd$ is a differentiable map, and therefore a trivialisation of the fibration $\mathbf{H}_\Delta \rightarrow \Hd$. Indeed, $R_C$ is given by the lift of the argument map to its universal covering on $C^\circ \setminus \R C^\circ$ and extended by the logarithmic Gauss map on $\R C^\circ$. Both maps depend analytically on the coefficients of the defining polynomial $f \in \mathbf{H}_\Delta $.

\end{proof}

We define a primitive integer segment to be a segment in $\Delta$ that joins two integer points and is of integer length 1. For a primitive integer segment $\sigma \subset \Delta$, denote by $\tilde{\sigma}$ its lift in $\tilde{\Delta}$.
It follows from the above proposition that for any primitive integer segment $\sigma \subset \Delta$, $(pr \circ R_C)^{-1}(\tilde{\sigma}) \subset C$ is a loop invariant by complex conjugation. Similarly, denote by $\tilde{v} \subset \tilde{\Delta}$ the boundary circle projecting down to $v \in \da \cap \Z^2$. Then $(pr \circ R_C)^{-1}(\tilde{v}) \subset \R C$ is one of the A-cycles of $C$.

\begin{Corollary}\label{cor:triv}
Let $C_0, \, C_1 \in \Hd$, $f: C_0 \rightarrow C_1$ be the diffeomorphism induced by the trivialisation $R$, $\sigma \subset \Delta$ a primitive integer segment and $v \in \da \cap \Z^2$. Then, the pullback by $f$ of the Dehn twist along the loop $(pr \circ R_{C_1})^{-1}(\tilde{\sigma})$ (respectively $(pr \circ R_{C_1})^{-1}(\tilde{v})$) in $C_1$ is the Dehn twist along the loop $(pr \circ R_{C_0})^{-1}(\tilde{\sigma})$ (respectively $(pr \circ R_{C_0})^{-1}(\tilde{v})$)  in $C_0$.\qed
\end{Corollary}

\begin{Definition}\label{def:segtwist}
For any primitive integer segment $\sigma \subset \Delta$ and any curve $C\in \Hd$, define the loop $\delta_\sigma := (pr \circ R_{C})^{-1}(\tilde{\sigma}) \subset C$ and $\tau_\sigma \in \MCG(C)$ to be the Dehn twist along  $\delta_\sigma$.\\
For any $v \in \da \cap \Z^2$, define the A-cycle $\delta_v := (pr \circ R_{C})^{-1}(\tilde{v}) \subset \R C$ and $\tau_v \in \MCG(C)$ to be the Dehn twist along  $\delta_v$.
\end{Definition}

\begin{Remark}
$\delta_v$ and $\delta_\sigma$ intersect if and only if $v$ is an end point of $\sigma$. In this case, $\delta_v$ intersects $\delta_\sigma$ transversally at one point.
\end{Remark}

\section{Tropical curves}\label{sec:tropi}

In this section, we recall some definitions on tropical and phase-tropical curves, following \cite{IMS} and \cite{L}. For the sake of pragmatism and simplicity, some definitions are given in a restrictive context. The expert reader should not be disturbed. The main goal of this section is the statement and proof of Theorem \ref{thm:rea}, as a corollary of Mikhalkin's approximation Theorem, see \cite[Theorem 5]{L}. Namely, we construct explicit elements in the image of the monodromy map $\mu$ by approximating well chosen loops in the relevant moduli space of phase-tropical curves.

A tropical polynomial in two variables $x$ and $y$ is a function
\begin{equation}\label{eq:troppol}
f(x,y) = ``\sum_{(\alpha, \beta) \in A} c_{\alpha, \beta} x^\alpha y^\beta" = \max_{(\alpha, \beta) \in A} (c_{\alpha, \beta} +x \alpha + y \beta)
\end{equation}
where $A \subset \N^2$ is a finite set. 

\begin{Definition}
The tropical zero set of a tropical polynomial $f(x,y)$ is defined by 
\[ \Gamma_f := \left\lbrace (x,y) \in \R^2 \, \vert \, \text{f is not smooth at } (x,y) \right\rbrace \]
A subset $\Gamma \subset \R^2$ is a tropical curve if $\Gamma = \Gamma_f$ for some tropical polynomial $f$. We denote  by $E(\Gamma)$ (respectively $V(\Gamma)$) the set of its bounded edges (respectively of its vertices).
\end{Definition}

It follows from the definition that a tropical curve $\Gamma \subset \R^2$ is a piecewise linear graph with rational slopes. The Newton polygon of a tropical polynomial $f$ is defined as the convex hull of its support $A$. The Newton polygon of a tropical curve is only defined up to translation as the multiplication of tropical polynomial by a tropical monomial does not affect its zero set. Tropical curves are intimately related to subdivision on their Newton polygon. For a finite set $A \subset \Delta\cap \N^2$, denote by $\Delta$ the lattice polygon $\Delta:=conv(A)$. For any function $h : A \rightarrow \R$, define
\[ \Delta_h := conv \left\lbrace \big( (\alpha, \beta) , t \big) \in \mathbb{R}^3  \: \vert \:  (\alpha, \beta) \in A, \: t \geq h (\alpha, \beta) \right\rbrace, \] 
$i.e.$ $\Delta_h$ is the convex hull of the epigraph of $h$. The bounded faces of $\Delta_h$ define a piecewise-linear convex function $\nu_h : \Delta \rightarrow \R$. Finally, define $S_h$ to be the subdivision $\Delta = \Delta_1 \cup \Delta_2 \cup \, ... \,  \cup \Delta_N $ given by the domains of linearity $\Delta_i$ of $\nu_h$. The subdivision $S_h$ is a union of 0-, 1- and 2-cells. For practical matters, we will often consider $S_h$ as a graph. The set of its vertices is $V(S_h)= \cup_i (\partial \Delta_i \cap \N^2)$ and the set of its edges $E(S_h)$ is the union over $i$ of all the primitive integer segments contained in $\partial \Delta_i$.

\begin{Definition}\label{def:subdiv}
A convex subdivision $S$ of $\Delta$ is a graph $S=S_h$ for some \linebreak $h : A \rightarrow \R$. The subdivision $S$ is unimodular if any connected component of $\Delta \setminus S$ has Euclidean area 1/2. 
\end{Definition}

Notice that by Pick's formula, a subdivision is unimodular if it decomposes $\Delta$ in triangles, whose vertices generates the lattice $\Z^2$.\\
Now any tropical polynomial $f$ as in (\ref{eq:troppol}) can be considered as the function $(\alpha, \beta) \mapsto c_{\alpha, \beta}$ on $A$, and hence induces a convex subdivision $S_f$ on its Newton polygon $\Delta$. One has the following duality, see \cite{IMS}.

\begin{Proposition}
Let $f(x,y)$ be a tropical polynomial. The subdivision $(\R^2,\Gamma_f)$ is dual to the subdivision $(\Delta,S_f)$ in the following sense
\begin{enumerate}
\item[$\ast$] $i$-cells of $(\R^2,\Gamma_f)$ are in 1-to-1 correspondence with $(2-i)$-cells of $(\Delta,S_f)$, and the linear spans of corresponding cells are orthogonal to each others.
\item[$\ast$] The latter correspondence reverses the incidence relations.
\end{enumerate}
For any edge $e\in E(\Gamma_f)$, denote its dual edge by $e^\vee \in E(S_f)$. Define the order of any connected component of $\R^2 \setminus \Gamma_f$ to be its corresponding lattice point in $\Delta$.
\end{Proposition}

\begin{Remark}\label{rem:popdec}
For any curve $C \in \Hd$, a unimodular convex subdivision $S$ of $\Delta$ induces a pair of pants decomposition of $C$ by considering the union of the loops $\delta_\sigma$ for all $\sigma \in E(S)$.
\end{Remark}

\begin{Definition}
A tropical curve $\Gamma_f \subset \mathbb{R}^2$ is smooth if its dual subdivision $S_f$ is unimodular.
\end{Definition}

In particular, a smooth tropical curve has only 3-valent vertices. Now, we introduce some material in order to define smooth phase-tropical curves. Consider the family of diffeomorphism on $\ttor$ given by 
\[H_t(z,w) = \left( \vert z \vert^{\frac{1}{\log(t)}} \frac{z}{\vert z \vert} , \vert w \vert^{\frac{1}{\log(t)}} \frac{w}{\vert w \vert} \right)\]
Define the phase-tropical line $L \subset \ttor$ to be the Hausdorff limit of \linebreak$H_t\big( \left\lbrace 1+z+w=0 \right\rbrace\big)$ when t tends to $+\infty$. If we denotes by $\Lambda := \Gamma_{``1+x+y"}$ the tropical line centred at the origin, $L$ is a topological pair of pants such that $\A(L)=\Lambda$. $L$ is given by $\left\lbrace z=0 \right\rbrace$, $\left\lbrace w=0 \right\rbrace$ and $\left\lbrace z=w \right\rbrace$ over the three open rays of $\Lambda$. These cylinders are glued to the coamoeba $Arg\big( \left\lbrace 1+z+w=0 \right\rbrace \big)$ over the vertex of $\Lambda$.\\
A toric transformation $ A : \ttor \rightarrow \ttor $ is a diffeomorphism of the form 
\[ (z,w) \mapsto \big( b_1 z^{a_{11}}w^{a_{12}}, b_2 z^{a_{21}}w^{a_{22}} \big) \]
where $(b_1,b_2) \in  \big( \mathbb{C}^\ast \big)^2 $ and $  \big( a_{ij} \big) \in Sl_{2} (\Z)$.
It descends to an affine linear transformation on $\mathbb{R}^2$ (respectively on $(S^1)^2$) by composition with the projection $\A$ (respectively $Arg$) that we still denote by $A$. Notice that the group of  toric transformations fixing $L$ is generated by $(z,w) \mapsto (w,z)$ and $(z,w)\mapsto(w,(zw)^{-1})$. It descends via $\A$ to the group of symmetries of $\Lambda$.

\begin{Definition}
A smooth phase-tropical curve $V \subset \ttor$ is a topological surface such that :
\begin{enumerate}
\item[$\ast$] its amoeba $ \mathcal{A} (V) $ is a smooth tropical curve $\Gamma \subset \mathbb{R}^2$,
\item[$\ast$] for any open set $U\subset \R^2$ such that $U\cap\Gamma$ is connected and contains exactly one vertex $v$ of $\Gamma$, there exists a toric transformation $A$ such that $V \cap \A^{-1}(U)=A(L)\cap \A^{-1}(U)$.
\end{enumerate}
We will say that the map $A$ is a chart of $V$ at the vertex $v$.
\end{Definition}

\begin{Remark}
It follows from the definition that the topology of $V$ is determined by its underlying tropical curve $\A(V)$. $V$ is an oriented genus $g$ surface with $b$ punctures where $g$ and $b$ are respectively  the genus and the number of infinite rays of $\A(V)$.
\end{Remark}

Smooth \ph curves enjoy a  description very similar to the one of Riemann surfaces in terms of Fenchel-Nielsen coordinates, see \cite{L}. Let $V$ be a smooth \ph curve and $\Gamma := \A(V)$. Up to toric translation, $V$ can be encoded by the pair $(\Gamma, \Theta)$ where $\Theta : E(\Gamma) \rightarrow S^1$ is defined in the following way.
For any $e\in E(\Gamma)$ bounded by $v_1, \, v_2 \in V(\Gamma)$, consider any two charts $A_1, \, A _2 : \ttor \rightarrow \ttor$ at $v_1$ and $v_2$ overlapping on $e$ such that $A_1$ and $A_2$ map $e$ on the same edge of $\Lambda$ and exactly one of the $A_i : \R^2 \rightarrow \R^2$ is orientation preserving. The induced automorphism $A_2\circ A_1^{-1}$ on the cylinder $\left\lbrace w=0\right\rbrace$ is given in coordinates by $z \mapsto \vartheta/z$ for a unique $\vartheta \in \C^\ast$. We can check that the quantity $arg(\vartheta)$ is an intrinsic datum of $V$, see \cite[Proposition 2.36]{L}.

\begin{Definition}
For a \ph curve $V\subset\ttor$ and $\Gamma := \A(V)$, the associated twist function $\Theta : E(\Gamma) \rightarrow S^1$ is defined by $$ \Theta(e):= arg(\vartheta)$$ for any $e\in E(\Gamma)$ and $\vartheta \in \C^\ast$ constructed as above.
\end{Definition}

\begin{Remark}\label{rem:phcurv}
If any smooth \ph curve $V \subset \ttor$ can be described by a pair $(\Gamma, \Theta)$, there are conditions on the twist function $\Theta$ for the pair $(\Gamma, \Theta)$ to correspond to a \ph curve in $\ttor$. These conditions are given by the equations $(5)$ in \cite{L}.
Notice also that any tropical curve $\Gamma$ induces a length function $l$ on $E(\Gamma)$.
The Fenchel-Nielsen coordinates of a \ph curve $(\Gamma, \Theta)$ consist in the pair $(l,\Theta)$, see \cite{L}.
\end{Remark}

Phase tropical Harnack curves are the \ph counterpart of the simple Harnack curves that we reviewed in the previous section.

\begin{Definition}
A \ph Harnack curve is a smooth phase-tropical curve $V \subset \ttor$ invariant by complex conjugation and with associated twist function $\Theta\equiv 1$.
\end{Definition}

It follows from the definition that for a given smooth tropical curve $\Gamma$, there are exactly four \ph Harnack curves $V \subset \ttor$ such that $\A(v)=\Gamma$ and they are all obtained from each others by sign changes of the coordinates. The map $\A$ sends the real part $ \R V$ onto $\Gamma$ in a $2$-to-$1$ so that the topological type of $(\R \Xd, \R V)$ can be recovered from $\Gamma$. If $\Delta$ is the Newton polygon of $\Gamma$, it follows from \cite[Theorem 2]{L2} that the approximation of $\Gamma$ in $\Li_\Delta$ is a simple Harnack curve whose amoeba can be made arbitrarily close to $\Gamma$.

We now introduce some necessary terminology for the statement of Theorem \ref{thm:rea}.

\begin{Definition}
A weighted graph of $\Delta$ is a pair $(G,m)$ where $G$ is a graph whose edges are primitive integer segments of $\Delta$ and $m : E(G) \rightarrow \Z$. A weighted graph $(G,m)$ is balanced if for any vertex $v\in V(G)\cap \itr(\Delta)$
\begin{equation}  
\sum_{\vec{e}\in E(G,v)} m(e)\cdot\vec{e}=0 \label{eq:balance}
\end{equation}
where $E(G,v)\subset E(G)$ is the set of edges adjacent to $v$ oriented outwards. A weighted graph $(G,m)$ is admissible if it is balanced and if it is a subgraph of a unimodular convex subdivision of $\Delta$.
\end{Definition}

\begin{Definition}
For a weighted graph $\G:=(G,m)$ of $\Delta$, and any curve $C_0\in\Hd$, define
$$\tau_\G:=\prod_{\sigma\in E(G)} \tau_\sigma^{m(\sigma)} \in \MCG(C_0)$$
where $\tau_\sigma$ is given in Definition \ref{def:segtwist}.
\end{Definition}

\begin{Theorem}\label{thm:rea}
For any curve $C_0 \in \Hd$ and any admissible graph $\mathcal{G}=(G,m)$ of $\Delta$, one has
\[  \tau_\mathcal{G}  \in \im(\mu). \]
\end{Theorem}

\begin{proof}
Let $S$ be a unimodular convex subdivision containing $G$ and $\Gamma \subset \R^2$ be a tropical curve of dual subdivision $S$. 
Among the four \ph Harnack curves supported on $\Gamma$, consider the one $V$ that is coherent with the sign convention made for 
$\Hd$. 

The admissible graph $\G=(G,m)$ induces the following family \linebreak $\left\lbrace V_\theta :=(\Gamma, \Theta_\theta) \, \vert \,0\leq 
\theta\leq 2 \pi \right\rbrace$ of \ph curves: for any $\epsilon \in E(\Gamma)$ such that $\epsilon^\vee \in E(G)$, define 
$\Theta_\theta(\epsilon) = e^{i\theta m(\epsilon^\vee)}$ and  $\Theta_\theta(\epsilon) = 1$ otherwise. One easily check that 
$\Theta_\theta$ satisfies $(5)$ in \cite{L} for any parameter $\theta$, so that it actually defines a loop of smooth \ph curves in $\ttor$ 
based at $V_0=V_{2\pi}=V$.

We now want to use Mikhalkin's approximation theorem in family (see \cite[Theorem 5]{L}). For this, we 
briefly reproduce a construction that can be found in \cite[Section 4.2]{L}. Consider the Teichm\"{u}ller space $\mathcal{T}(C)$ pointed at 
an arbitrary curve $C \in \Hd$. The pair of pants decomposition given by $\left\lbrace \delta_\sigma \, \vert \, \sigma \in E(S)  
\right\rbrace$ induces Fenchel-Nielsen coordinates on $\mathcal{T}(C)$. The partial quotient $\mathcal{T}(C) \rightarrow 
(\C^\ast)^{E(\Gamma)} \rightarrow \mathcal{M}_{g_\Li,b_\Li}$ by the group generated by the set of Dehn twists $\left\lbrace \tau_\sigma 
\, \vert \, \sigma \in E(S) \right\rbrace$ inherits Fenchel-Nielsen coordinates $\R_{>0}^{E(\Gamma)} \times (S^1)^{E(\Gamma)} \simeq 
(\C^\ast)^{E(\Gamma)}$. Using these coordinates, we consider the partial compactification $F$ of $(\C^\ast)^{E(\Gamma)}$ obtained by 
considering the real oriented blow-up of its $\R_{>0}$-factor at the origin and identify its boundary points  with \ph curves of 
coordinates $(l,\Theta)$ up to positive rescaling of $l$, see Remark \ref{rem:phcurv}. Hence, the 1-parametric family $\left\lbrace V_\theta \right\rbrace_\theta$ induces 
a loop in $\partial F$. 

By \cite[Proposition 4.9]{L}, there exists a smooth analytic subset $U \subset \R_{>0}^{E(\Gamma)} \times 
\left\lbrace \Theta_\theta \right\rbrace_\theta$ of codimension $g_{\Li}$ such that $\overline{U} \subset F$ 
intersects transversally $\partial F$ along a smooth locus containing $\left\lbrace V_\theta 
\right\rbrace_\theta$. By the same proposition, $U$ comes with a continuous map $s : U \rightarrow \PSec$ given by integration of particular 
differentials such that the curves $u\in U$ and $s(u)\in \PSec$ are isomorphic. It implies that one can find a closed path $\left\lbrace 
(l_\theta, \Theta_\theta) \right\rbrace_\theta \subset U$ homotopic to $\left\lbrace V_\theta \right\rbrace_\theta$ in $\overline{U}$. Define $C^\mathbb{T}_\theta:= s(l_\theta, \Theta_\theta)$. According to \cite[Theorem 2]{L2}, 
$\left\lbrace (l_\theta, \Theta_\theta)\right\rbrace_\theta$ can be chosen so that 
$C^\mathbb{T}_0=C^\mathbb{T}_{2\pi} \in \Hd$. Moreover, it follows from the construction of the map $s$ that the pair of pants 
decomposition of $C^\mathbb{T}_0$ induced by $s$ is the decomposition given by $\left\lbrace \delta_\sigma \, \vert \, \sigma \in E(S) 
\right\rbrace$. Together with the definition of $\Theta_\theta$, it implies that the monodromy induced by the closed path $\left\lbrace 
C_\theta^\mathbb{T}  \right\rbrace_\theta$  is 
$$ \prod_{\sigma\in E(G)} \tau_{\sigma}^{m(\sigma)} = \tau_\mathcal{G} \in \MCG(C^\mathbb{T}_0).$$
Now, By the connectedness of $\Hd$, see Theorem \ref{thm:connected}, we can conjugate the closed path $\left\lbrace 
C^\mathbb{T}_\theta \right\rbrace_{\theta}$ by a continuous path in $\Hd$ joining $C_0$ to $C^\mathbb{T}_0$ in order to get  a closed 
path $\epsilon$ based at $C_0$. By Corollary \ref{cor:triv}, we deduces that $ \mu(\epsilon)=\tau_\mathcal{G} \in \MCG(C_0)$.
\end{proof}

\section{Combinatorial tools}\label{sec:combi}

\subsection{Unimodular convex subdivisions}

In order to apply Theorem \ref{thm:rea}, we need to be able to construct unimodular convex subdivisions of a polygon. We give an extension lemma and a refinement lemma which will be our main construction tools.

\begin{Lemma}\label{lem:extend}
Let $\Delta'\subset \Delta$ be two convex lattice polygons and take a convex subdivision $S'$ of $\Delta'$. There exists a convex subdivision of $\Delta$ which extends $S'$.
\end{Lemma}

\begin{proof}
  By assumption, we can associate a piecewise-linear convex function $\nu : \Delta'\rightarrow \R$ to the subdivision $S'$. We may also assume that $\nu$ only takes positive values. Denote by $K$ the set of vertices  of $\Delta$ which are not in $\Delta'$. For each $h\in\R$, define the piecewise linear function $\nu_h: \Delta\rightarrow \R$ whose graph is the union of the bounded faces of the convex hull of the union of the epigraph of $\nu$ and of the rays $K\times [h,+\infty]$.

There exists a value $h_0\in\R$ such that for any $h\geq h_0$, the subdivision $S_{\nu_h}$ of $\Delta$ associated to $\nu_h$ and the subdivision $S$ coincide on the interior of $\Delta'$. Indeed, we can take $h_0$ such that the set $K\times \{h_0\}\subset \Delta\times \R$ is above all the supporting hyperplanes of the epigraph of $\nu$.

Now we only need to ensure that we can choose $h\geq h_0$ such that the boundary of $\Delta'$ is supported on $S_h$. For each pair of lattice points $v$, $w$ which are consecutive on the boundary of $\Delta'$ and not both in the boundary of $\Delta$, consider the points $P_{\pm\varepsilon} = \frac{v+w}{2} \pm \varepsilon n$, where $n$ is the exterior normal unit vector to $\Delta'$ on the segment $[v,w]$. We can take an $\varepsilon_0>0$ such that for any $0<\varepsilon\leq\varepsilon_0$ and for any $h\geq h_0$, the segment $[P_{-\varepsilon},P_{\varepsilon}]$ intersects the subdivision $S_{\nu_h}$ at most once. This intersection is nonempty exactly when the segment $[v,w]$ is in $S_{\nu_h}$. Since $\Delta'$ has only a finite number of lattice points, we can repeat the same construction for all such pairs of points $v$ and $w$ to obtain a finite set of points $P^i_{\pm \varepsilon_0}$, $i\in I$, and a number $\varepsilon>0$ such that for all $i\in I$ and $h\geq h_0$, the segment $[P^i_{-\varepsilon},P^i_{\varepsilon}]$ intersects the subdivision $S_{\nu_h}$ at most once. The restriction of $\nu$ over the segment $[P^i_{-\varepsilon},P^i_{\varepsilon}]\cap\Delta'$ is affine linear of slope $\alpha_i$. We need to find $h\geq h_0$ such that $\nu$ is not affine linear along the whole segment $[P^i_{-\varepsilon},P^i_{\varepsilon}]$.

For any $i\in I$ and any $h\geq h_0$, the point $(P^i_{\varepsilon},\nu_h(P^i_{\varepsilon}))\in\Delta\times\R$ is a convex combination $\sum_{A\in\Delta'\cap\Z^2} \lambda_{A}(A,\nu_h(A)) + \sum_{B\in K} \mu_{B}(B,h)$. The coefficients $\lambda_A\in [0,1]$ and $\mu_B\in [0,1]$ are not uniquely defined. However, since the points $P^i_{\varepsilon}$ are outside $\Delta'$, there exists a $\mu_0>0$ such that for any such convex combination giving some point $P^i_{\varepsilon}$, there exists $B\in K$ with $\mu_B\geq \mu_0$. Take $h = \max(h_0,\frac{\max(\nu) + 2\varepsilon \max_{i\in I}(\alpha_i)}{\mu_0} +1)$. Then, the function $\nu_h$ cannot be affine linear along any of the segments $[P^i_{-\varepsilon},P^i_{\varepsilon}]$ as its value at the points $P^1_{\varepsilon}$ is too high. Thus, the subdivision $S_{\nu_h}$ contains the boundary of $\Delta'$ and extends $S'$.
\end{proof}

\begin{Lemma}[\cite{Haa}, Lemma 2.1]\label{lem:refinement}
  Let $S$ be a convex subdivision of a convex lattice polygon $\Delta$. There exists a refinement of $S$ which is a unimodular convex subdivision of $\Delta$.
\end{Lemma}

\subsection{Constructing Dehn twists}

We give some applications of Theorem \ref{thm:rea} that will be useful in the rest of the paper. In particular, the following constructions are enough to treat the case $d<2$. Let $\Delta$ be a smooth convex lattice polygon and $(X,\Li)$ the associated polarized toric surface. Fix a curve $C_0\in\Hd$.

\begin{Lemma}\label{lem:chasing}
Let $\G = (G,m)$ be a weighted graph of $\Delta$ such that $\tau_{\G}\in\im(\mu)$ (resp. $[\tau_{\G}]\in\im([\mu])$). Assume that $\G$ has a vertex $v\in V(G)$ of valency $1$ which is an interior point of $\Delta$ and let $\sigma\in E(G)$ be the edge ending at $v$. Denote by $\G'$ the weighted graph obtained from $\G$ by removing $\sigma$. If $m(\sigma) = \pm 1$ then $\tau_{\sigma}\in\im(\mu)$ and $\tau_{\G'}\in\im(\mu)$ (resp. $[\tau_{\sigma}]\in\im([\mu])$ and $[\tau_{\G'}]\in\im([\mu])$).  
\end{Lemma}

\begin{proof}
Assume first that $m(\sigma)=1$. We know from Theorem \ref{thm:acycle} that $\tau_v$ is in $\im(\mu)$. Since $\delta_v$ and $\delta_{\sigma}$ intersect transversely in only one point, $\tau_v\tau_{\sigma}(\delta_v)$ is isotopic to $\delta_{\sigma}$. Moreover $\tau_v$ and $\tau_{\sigma}$ commute with $\tau_{\G'}$. Thus we obtain
\[
(\tau_v\tau_{\G})\tau_v(\tau_v\tau_{\G})^{-1} = (\tau_v\tau_{\sigma}\tau_{\G'})\tau_v(\tau_v\tau_{\sigma}\tau_{\G'})^{-1} = (\tau_v\tau_{\sigma})\tau_v(\tau_v\tau_{\sigma})^{-1} = \tau_{\sigma}\in\im(\mu).
\]
The case when $m(\sigma)=-1$ is immediate.

Since the map $\MCG(C_0)\rightarrow \spaut(H_1(C_0,\Z))$ is a morphism, the homological statement follows from the previous argument.
\end{proof}

\begin{Notation}
In the rest of the paper, we will draw weighted graphs using two different line styles: the dashed lines will be for segments for which we know that the associated Dehn twist is in the image of the monodromy and the plain lines will be for the others.
\end{Notation}

Let us introduce a type of primitive integer segment that will appear often in the rest of the paper.

\begin{Definition}\label{def:bridge}
A primitive integer segment $\sigma\subset \Delta$ is a bridge if it joins a lattice point of $\partial \Delta$ to a lattice point of $\partial \da$ and  does not intersect $\itr (\da)$.
\end{Definition}

\begin{Lemma}\label{lem:allbridge}
Let $v$ be a lattice point on $\partial \da$ and $\sigma'$ be a bridge ending at $v$.
Assume that $(\tau_{\sigma'})^m\in\im(\mu)$ (resp. $[\tau_{\sigma'}]^m\in\im([\mu])$) for some integer $m\in \Z$.  Then we have $(\tau_{\sigma})^m\in\im(\mu)$ (resp. $[\tau_{\sigma}]^m\in\im([\mu])$) for any bridge $\sigma$ ending at $v$.
\end{Lemma}

\begin{proof}
This follows from the fact that for any two bridges $\sigma$ and $\sigma'$ ending at the same point of $\da$, the loops $\delta_{\sigma}$ and $\delta_{\sigma'}$ are isotopic in $C_0$.
\end{proof}

\begin{Proposition}\label{prop:corner}
  Let $\kappa\in\da$ be a vertex. If $\sigma$ is a bridge ending at $\kappa$, then $\tau_{\sigma}$ is in $\im(\mu)$.
\end{Proposition}

\begin{proof}
  Using an invertible affine transformation preserving the lattice, we may assume that $\kappa$ is the point $(1,1)$, and that the origin is a vertex of $\Delta$ and the two edges of $\Delta$ meeting at this point are directed by $(1,0)$ and $(0,1)$.

Using Lemma \ref{lem:allbridge}, we may assume that $\sigma$ is the segment joining $(0,0)$ to $(1,1)$. Denote by $\sigma_1$ and $\sigma_2$ the segments joining $(1,1)$ respectively to $(1,0)$ and $(0,1)$. Define the weighted graph $\G=(G,m)$ given by Figure \ref{fig:corner}.

\begin{figure}[h]
\centering{
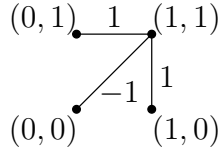}
\caption{The weighted graph $\G$} \label{fig:corner}
\end{figure}

 It satisfies the balancing condition at the point $(1,1)$. To prove it is admissible, we only need to find a unimodular convex subdivision of $\Delta$ supporting $G$. To that end, take the subdivision of the square consisting of $G$ and the segments joining $(0,0)$ to $(1,0)$ and $(0,0)$ to $(0,1)$. It is convex as it can be obtained from the function taking value $0$ on $(0,0)$ and $(1,1)$ and $1$ on $(0,1)$ and $(1,0)$. Using Lemmas \ref{lem:extend} and \ref{lem:refinement}, we obtain a unimodular convex subdivision of $\Delta$ which contains $G$.

Thus $\G$ is an admissible graph of $\Delta$. We apply Theorem \ref{thm:rea} to deduce that $\tau_{\G}$ is in $\im(\mu)$. However, the loops $\delta_{\sigma}$, $\delta_{\sigma_1}$ and $\delta_{\sigma_2}$ are all isotopic. It follows that $\tau_{\G} = \tau_{\sigma}$ which proves the statement.
\end{proof}

\begin{Proposition}\label{prop:side}
  Let $\sigma$ be a primitive integer vector lying on an edge of $\da$. The Dehn twist $\tau_{\sigma}$ is in $\im(\mu)$.
\end{Proposition}

\begin{proof}
 By Proposition \ref{prop:poly}, we may use an invertible affine transformation of $\R^2$ preserving the lattice such that $\Delta$ has a vertex at $(-1,-1)$ with corresponding edges going through $(-1,0)$ and $(0,-1)$ and such that the edge of $\da$ containing $\sigma$ is the segment joining $(0,0)$ to $(l,0)$ for some $l\geq 1$. Notice that the points $(-1,0)$ and $(l+1,0)$ are on the boundary of $\Delta$. For $i$ from $0$ to $l+1$, denote by $\sigma_i$ the segment joining $(i-1,0)$ to $(i,0)$.

We know from Proposition \ref{prop:corner} that $\tau_{\sigma_0}$ and $\tau_{\sigma_{l+1}}$ are in $\im(\mu)$. Define the weighted graph $\G = (G,m)$ given by Figure \ref{fig:side}.

\begin{figure}[h]
\centering{
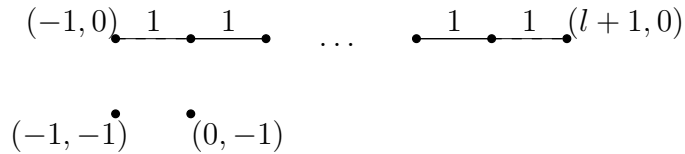}
\caption{The weighted graph $\G$} \label{fig:side}
\end{figure}

This graph is balanced at all the vertices in $\da$. On the other hand, we can take a convex function on $\Delta$ which is affine linear on each component of $\Delta\setminus G$. The associated convex subdivision contains $G$. Using Lemma \ref{lem:refinement}, it follows that $\G$ is an admissible graph.

From Theorem \ref{thm:rea}, we deduce that $\tau_{\G} = \prod\tau_{\sigma_i}\in\im(\mu)$. We now apply Lemma \ref{lem:chasing} several times which shows that $\tau_{\sigma_{i}}\in\im(\mu)$ for all $i$ between $1$ and $l$.
\end{proof}

\subsection{The elliptic case}\label{sec:dless2}

\begin{proof}[Proof of Theorem \ref{thm:main1}, case $d=0$]
Let $\Xd$ be a smooth and complete toric surface and let $\Li$ be an ample line bundle on $\Xd$. Denote by $\Delta$ the associated lattice polygon and fix a curve $C_0\in\Hd$. Assume that $\Delta$ has only one interior lattice point $\kappa$ and take $\sigma$ to be one of the bridges ending at $\kappa$.  We know from Proposition \ref{prop:corner} and from Theorem \ref{thm:acycle} that both $\tau_{\kappa}$ and $\tau_{\sigma}$ are in $\im(\mu)$. Since $C_0$ is of genus $1$, the group $\MCG(C_0)$ is generated by $\tau_{\kappa}$ and $\tau_{\sigma}$. Thus, $\im(\mu)=\MCG(C_0)$.
\end{proof}

\section{Proof of Theorems \ref{thm:main1} and \ref{thm:main2}}\label{sec:dis2}

In the rest of the paper, $\Xd$ is a smooth and complete toric surface equipped with an ample line bundle $\Li$. The associated polygon is denoted by $\Delta$ and its interior polygon by $\da$. Since we already took care of the elliptic case in \S \ref{sec:dless2}, we assume that $\dim(\da) = 2$ and fix a curve $C_0\in\Hd$.

\subsection{Normalizations}\label{sec:normalization}

We first give a description of $\Delta$ near a vertex of $\da$.

\begin{Definition}\label{def:normalization}
If $\kappa$ is a vertex of $\da$, a normalization of $\Delta$ at $\kappa$ is an invertible affine transformation $A:\R^2\rightarrow\R^2$ preserving the lattice and mapping $\kappa$ to $(0,0)$ and its adjacent edges in $\da$ to the edges directed by $(1,0)$ and $(0,1)$.
\end{Definition}

We will usually abuse notation and keep calling $\Delta$ the image of $\Delta$ under a noramlization.

 It follows from Proposition \ref{prop:poly} that after a normalization at $\kappa$, there are two possibilities for the polygon $\Delta$ near $\kappa$, as depicted in the Figure \ref{fig:cornercase}.

\begin{figure}[h]
\centering{
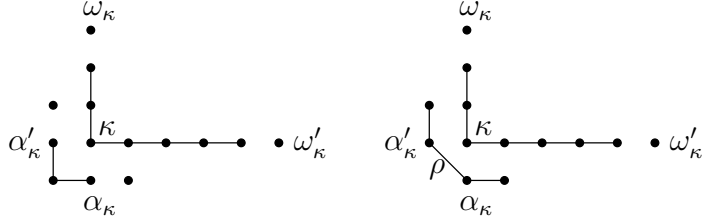}
\caption{The polygon $\Delta$ near $\kappa$} \label{fig:cornercase}
\end{figure}

Indeed, since the normal fan of $\Delta$ is a refinement of that of $\da$, there is an edge of $\Delta$ supported on the line $y=-1$ and one supported on the line $x=-1$. The two cases are as follows :
\begin{itemize}
\item either those two edges meet, then we have the left case of Figure \ref{fig:cornercase};
\item or there is at least one other edge between those two. Let us call it $\rho$. Since $\rho$ has no corresponding edge in $\da$, we have  $l_{\rho}-D_{\rho}^2-2=0$. Since $\dim(\da) = 2$, this can only happen if $l_{\rho}=1$ and $D_{\rho}^2=-1$ (see \cite{Koe}, Lemma 2.3.1). The assumption that $\da$ has dimension 2 also ensures that there cannot be two consecutive edges satisfying those two equalities. Geometrically, the primitive integer vector supporting $\rho$ has to be the sum of the two primitive integer vectors supporting its adjacent edges (see \cite[\S 2.5]{Ful}). We obtain the right case of Figure \ref{fig:cornercase}. 
\end{itemize}

 After a normalization of $\Delta$ at $\kappa$, it follows from the above discussion that the points $\ak=(0,-1)$ and $\akp=(-1,0)$ always belong to $\partial \Delta$. Moreover, if the edge of $\da$ directed by $(0,1)$ (resp. by $(1,0)$) is of length $l$ (resp. of length $l'$), then the point $\ok = (0,l+1)$ (resp. $\okp = (l'+1,0)$) belongs to $\partial \Delta$.

\subsection{Constructions}

Let us start with some technical definitions. Fix a vertex $\kappa$ of $\da$ and choose a boundary lattice point $\kappa'$ of $\da$ such that $\kappa$ and $\kappa'$ are consecutive on $\partial \da$ (the orientation does not matter). Denote by $\epsilon$ the edge of $\da$ containing $\kappa'$. Choose a bridge $\sigma'$ which joins $\kappa'$ to any boundary lattice point of $\Delta$.  We aim to prove the following.

\begin{Lemma}\label{lem:propagate}
Assume that $\tau_{\sigma'}\in\im(\mu)$ (resp. $[\tau_{\sigma'}]\in\im([\mu])$). Then for any bridge $\sigma \subset \Delta$, we have $\tau_{\sigma}\in\im(\mu)$ (resp. $[\tau_{\sigma}]\in\im([\mu])$).
\end{Lemma}

\begin{Remark} 
In the rest of Section \ref{sec:dis2}, we will make constant use of admissible graphs in the proofs of the various statements. Most of the time, we will only show that the weighted graphs under consideration are balanced and leave the proof of the admissibility for the reader. Note that we need only to exhibit convex subdivision supporting balanced graph $(\mathcal{G},w)$ and obtain unimodular ones by lemma \ref{lem:refinement}.
\end{Remark}

\begin{proof}[Proof of Lemma \ref{lem:propagate}]
Consider the normalization of $\Delta$ at $\kappa$ mapping $\kappa'$ to $(0,1)$. The edge $\epsilon$ is then vertical.

We claim the following: if $\tau_{\sigma'}\in\im(\mu)$ and $\kappa''$ is a lattice point on the edge of $\da$ adjacent to $\epsilon$ and containing $\kappa$, then there is a bridge $\sigma$ ending at $\kappa''$ such that $\tau_{\sigma}\in\im(\mu)$.

By Lemma \ref{lem:allbridge}, we can assume without loss of generality that $\sigma'$ joins $\kappa'$ to $\akp = (-1,0)$. By Proposition \ref{prop:corner}, there is nothing to prove if $\kappa''$ is a vertex of $\da$. Therefore we can assume that $\kappa''=(a,0)$ is not a vertex. Now consider the admissible graph $\mathcal{G}$ of Figure \ref{fig:propag1}, where the horizontal edges have weight $-2a$ and the vertical ones have weight $-a-1$. By Theorem \ref{thm:rea}, we have $\tau_\mathcal{G} \in \im(\mu)$.
By Proposition \ref{prop:interior} applied to $\mathcal{G}$ and the assumption $\tau_{\sigma'}\in\im(\mu)$, Lemma \ref{lem:chasing} allows us to chase any primitive integer segment of $\mathcal{G}$ except $\sigma$ and conclude that $\tau_{\sigma}\in\im(\mu)$. This proves the claim.

By applying the above claim and Lemma \ref{lem:allbridge} to any $\kappa''$, we show that $\tau_{\sigma}\in\im(\mu)$ for any bridge ending on the edge of $\da$ adjacent to $\epsilon$ and containing $\kappa$. By induction on the edges of $\da$, it follows that $\tau_{\sigma}\in\im(\mu)$ for any bridge of $\Delta$.

The homological statement follows from the previous arguments and the fact that Lemma \ref{lem:chasing} and Proposition \ref{prop:corner} hold in homology.
\end{proof}

\begin{figure}[h]
\centering{
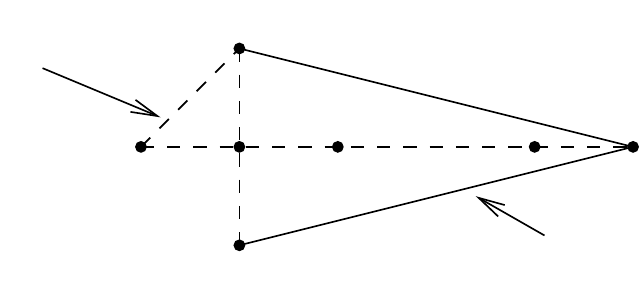}
\caption{The admissible graph $\G$.}
\label{fig:propag1}
\end{figure}

\begin{Remark}\label{rem:corner2}
If $\da$ has a side of integer length 1, it follows from Lemma \ref{lem:propagate} and Proposition \ref{prop:corner} that $\tau_{\sigma}\in\im(\mu)$ for any bridge $\sigma\subset\Delta$.
\end{Remark}

In the rest of this section, we explain how to construct a Dehn twist around a loop corresponding to a primitive integer segment $\sigma$ joining any two points of $\Delta$ (see Proposition \ref{prop:interior}). To that end, we will construct an admissible graph containing $\sigma$ with multiplicity $1$ and such that the Dehn twists corresponding to the edges of the graph adjacent to $\sigma$ are in $\im(\mu)$. We will construct this graph by parts.

Let us choose again a vertex $\kappa$ of $\da$ and a boundary lattice point $\kappa'$ of $\da$ such that $\kappa$ and $\kappa'$ are consecutive on $\partial \da$. Take a point $v\in\itr(\da)\cap\Z^2$ and choose two integers $m_1$ and $m_2$. We construct two weighted graphs \linebreak $\G_{\kappa,\kappa',v}^{m_1,m_2} = (G_{\kappa,\kappa',v},m_{\kappa,\kappa',v}^{m_1,m_2})$ and $\G_{\kappa',\kappa,v}^{m_1,m_2} = (G_{\kappa',\kappa,v},  m_{\kappa',\kappa,v}^{m_1,m_2})$ associated to these data as follows.

First, consider once again the normalization of $\Delta$ at $\kappa$ mapping $\kappa'$ to $(0,1)$. Let us also recall from \S \ref{sec:normalization} that the points $\ak=(0,-1)$ and $\akp=(-1,0)$ are boundary points of $\Delta$. The set of edges of $G_{\kappa,\kappa',v}$ is the union of two subsets, as well as that of $G_{\kappa',\kappa,v}$.

 The first subset is constructed by induction. We start with $G_{\kappa,\kappa',v}$. Take $v'_0 = v$ and define $v'_1$ as follows. We call $\sigma_{1,0}$ the primitive integer segment starting at $v'_0$ and directing $[v'_0,\kappa']$. Consider the ray starting at $v'_0$ and directed by $\sigma_{1,0}$ and make it rotate counter-clockwise around $v'_0$ until it hits a lattice point in $conv(v_0',\kappa',\kappa)$. Denote by $l$ this new ray and by $\sigma_{2,0}$ the primitive integer segment directing it (see Figure \ref{fig:rotray}, on the left).

\begin{figure}[h]
\centering{
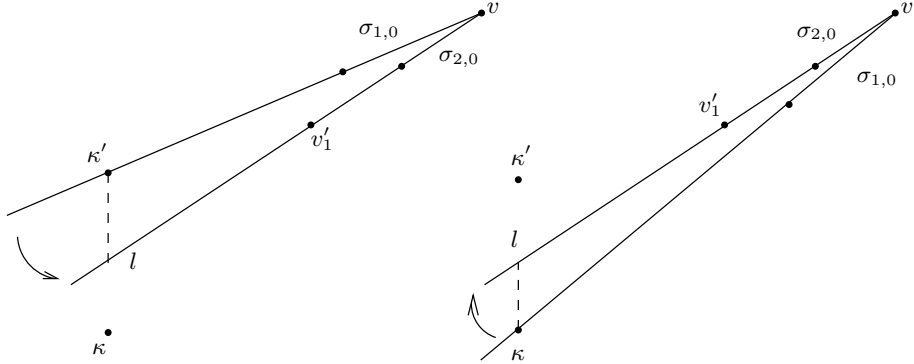}
\caption{Sketch of the construction of $v'_1$ for the graphs $G_{\kappa,\kappa',v}$ (on the left) and for the graph $G_{\kappa',\kappa,v}$ (on the right). There is no lattice point inside the triangles.} \label{fig:rotray}
\end{figure}

Take $v'_{1}$ to be the lattice point in $l\cap conv(v'_0,\kappa',\kappa)$ which is the farthest away from $v'_0$. Take all the primitive integer segments on the segment $[v'_0,\kappa']$ with weight $m_{1,0} = m_1$ and all the primitive integer segments on the segment $[v'_0,v'_1]$ with weight $m_{2,0} = m_2$ to be edges of $G_{\kappa,\kappa',v}$.

Suppose we have constructed $v'_0,\ldots,v'_k$ for some $k\geq 1$ and that the edges on $[v'_{k-1},\kappa']$ are weighted with $m_{1,k-1}$ and those on $[v'_{k-1},v'_k]$ have weight $m_{2,k-1}$. If $v'_k = \kappa$ we are done. If not, replacing $v'_0$ by $v'_{k}$ in the previous case, we construct $v'_{k+1}\in conv(v'_k,\kappa',\kappa)$ and primitive integer segments $\sigma_{1,k}$ on $[v'_k,\kappa']$ and $\sigma_{2,k}$ on $[v'_k,v'_{k+1}]$. By construction, $\sigma_{1,k}$ and $\sigma_{2,k}$ generate the lattice. Thus we can find integers $m_{1,k}$ and $m_{2,k}$ such that 
\[
m_{2,k-1}\sigma_{2,k-1} + m_{1,k}\sigma_{1,k} + m_{2,k}\sigma_{2,k} = 0.
\] 
Here we identify the segments with the vectors of $\Z^2$ they induce once we orient them going out of $v'_k$. We take all the primitive integer segments on the segments $[v'_k,\kappa']$ and $[v'_k,v'_{k+1}]$ weighted respectively by $m_{1,k}$ and $m_{2,k}$ to be edges of $G_{\kappa,\kappa',v}$.

 After a finite number of steps $t\geq 1$, we obtain $v'_t = \kappa$. The graph we obtain looks like the left graph in Figure \ref{fig:sketch}.

The construction for $G_{\kappa',\kappa,v}$ is the same except we exchange the roles of $\kappa$ and $\kappa'$; that is we use the ray starting at $v$ and going through $\kappa$ to sweep-out the triangle $v,\kappa,\kappa'$, and we stop when we encounter $\kappa'$ (see Figures \ref{fig:rotray} and \ref{fig:sketch}). We use the same notations for both graphs.

\begin{figure}[h]
\centering{
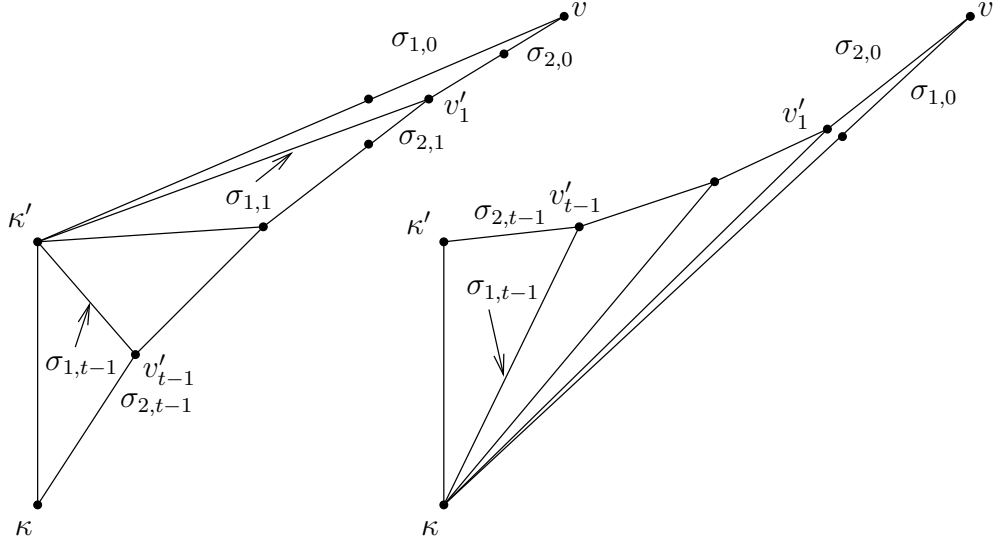}
\caption{Rough sketch of the graphs $G_{\kappa,\kappa',v}$ (on the left) and $G_{\kappa',\kappa,v}$ (on the right). There is no lattice point in any of the small triangles.}\label{fig:sketch}
\end{figure}

The second set of edges for both graphs consists of the four segments $\rho_1,\ldots,\rho_4$ joining respectively $\kappa'$ to $\akp=(-1,0)$ and $\kappa$ to $\kappa'$, $\akp$ and $\ak=(0,-1)$. The first two generate the lattice. Thus we can find two integers $a_1$ and $a_2$ such that $m_{1,0}\sigma_{1,0} + \ldots + m_{1,n-1}\sigma_{1,n-1} + a_1\rho_1 + a_2\rho_2 = 0$ (where the segments are oriented going out of $\kappa'$). We set $m_{\kappa,\kappa',v}^{m_1,m_2}(\rho_1) = a_1$ and $m_{\kappa,\kappa',v}^{m_1,m_2}(\rho_2) = a_2$. The segments $\rho_3$ and $\rho_4$ also generate the lattice. Hence, we can find two integers $a_3$ and $a_4$ such that $m_{2,n-1}\sigma_{2,n-1} + a_2 \rho_2 + a_3\rho_3 + a_4\rho_4 = 0$ (where the segments are oriented going out of $\kappa$). We set $m_{\kappa,\kappa',v}^{m_1,m_2}(\rho_3) = a_3$ and $m_{\kappa,\kappa',v}^{m_1,m_2}(\rho_4) = a_4$.

Notice that by construction the graphs $G_{\kappa,\kappa',v}$ and $G_{\kappa',\kappa,v}$ do not depend on $m_1$ and $m_2$ and that the weighted graph $\G_{\kappa,\kappa',v'_k}^{m_{1,k},m_{2,k}}$ (resp. $\G_{\kappa',\kappa,v'_k}^{m_{1,k},m_{2,k}}$) is a subgraph of $\G_{\kappa,\kappa',v}^{m_1,m_2}$ (resp. $\G_{\kappa',\kappa,v}^{m_1,m_2}$) for any $0\leq k < t-1$, although the weights on the segments $\rho_1, \ldots, \rho_4$ may differ from one to the other.

We give a complete example when we take $v=(16,4)$ and weights $m_1=m_2=1$ in Figure \ref{fig:halfdiamond}.

\begin{figure}[h]
\centering{
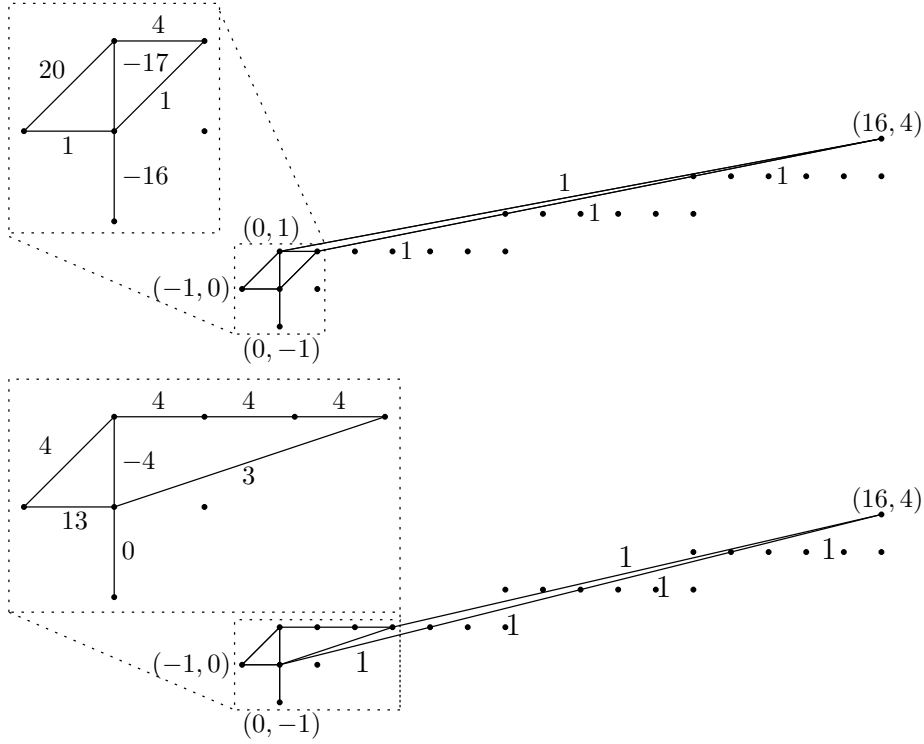}
\caption{The weighted graphs $\G_{(0,0),(0,1),(16,4)}^{(1,1)}$ (above) and $\G_{(0,1),(0,0),(16,4)}^{(1,1)}$ (under)} \label{fig:halfdiamond}
\end{figure}

\begin{Lemma}\label{lem:diamonds}
 The weighted graphs $\G_{\kappa,\kappa',v}^{m_1,m_2}$ and $\G_{\kappa',\kappa,v}^{m_1,m_2}$ satisfy the following properties:
  \begin{itemize}
  \item the edges $\sigma_1 = \sigma_{1,0}$ and $\sigma_2 = \sigma_{2,0}$ generate the lattice,
  \item the weight function $m_{\kappa,\kappa',v}^{m_1,m_2}$ (resp. $m_{\kappa',\kappa,v}^{m_1,m_2}$) satisfies the balancing condition \eqref{eq:balance} at every element of $V(G_{\kappa,\kappa',v})\cap\itr\Delta$ (resp. $V(G_{\kappa',\kappa,v})\cap\itr\Delta$) except at $v$.
  \end{itemize}
Assume moreover that there is a bridge $\sigma'$ ending at $\kappa'$ such that $\tau_{\sigma'}$ is in $\im(\mu)$ (resp. $[\tau_{\sigma'}]$ is in $\im([\mu])$). Then $\tau_{\sigma_1}$ and $\tau_{\sigma_2}$ are in $\im(\mu)$ (resp. $[\tau_{\sigma_1}]$ and $[\tau_{\sigma_2}]$ are in $\im([\mu])$).
\end{Lemma}

\begin{proof}
  The two properties of $\G_{\kappa,\kappa',v}^{m_1,m_2}$ and $\G_{\kappa',\kappa,v}^{m_1,m_2}$ follow from the construction.

For the second part of the Lemma,  we start with the weighted graph $\G_{\kappa,\kappa',v}^{m_1,m_2}$ and show how to adapt the proof for the other one. 

The strategy is to include the graph $G_{\kappa,\kappa',v}$ in another one called $G$ which will be supported by a unimodular convex subdivision of $\Delta$. Moreover, we will be able to put two weight functions on $G$ such that $\sigma_1$ appears with weight $1$ for one and $\sigma_2$ appears with weight $1$ for the other and such that both functions satisfy the balancing condition at every interior vertex of $G$. To conclude, we apply Theorem \ref{thm:rea} and Lemma \ref{lem:chasing}.

Let us give the details of the proof. Denote by $\xi$ the vertex of $\da$ next to $\kappa$ such that $\kappa'$ is not on $[\xi,\kappa]$ and take $\xi'$ to be the boundary point of $\da$ next to $\xi$ and outside $[\xi,\kappa]$. In the coordinates chosen previously, $\xi$ is of the form $(l,0)$ and $\xi'$ is given by $(l+s,1)$, for some integers $l$ and $s$. The points $(l-s,-1)$ and $(l+1,0)$ are boundary lattice points of $\Delta$.

 Define the graph $G$ as the union of $G_{\kappa,\kappa',v}$ and $G_{\xi,\xi',v}$ (see Figure \ref{fig:biggraph}).

\begin{figure}[h]
\centering{
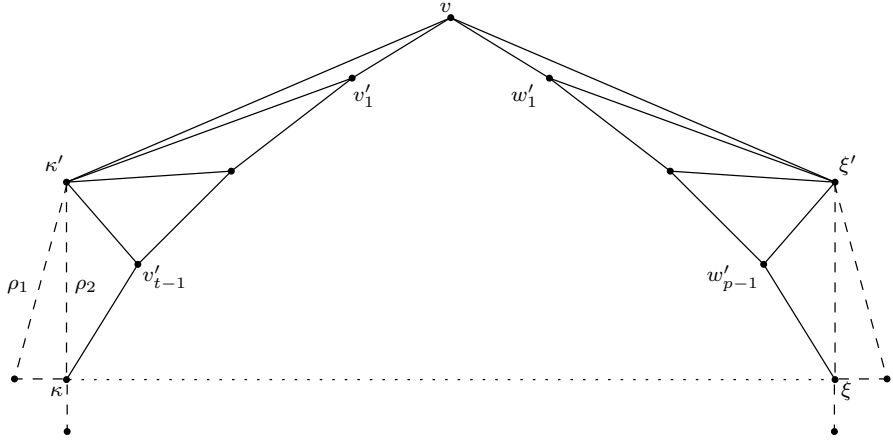}
\caption{Rough sketch of the graph $G$ (with $s=0$). There is no lattice point in the small triangles. The dashed lines indicate the edges of $G$ for which we know that the associated Dehn twist is in the image of the monodromy.} \label{fig:biggraph}
\end{figure}

 Let us show that it is supported by a unimodular convex subdivision of $\Delta$. First, we construct the subdivision on the pentagon $Q$ with vertices $\kappa',\kappa,\xi,\xi',v$ which is convex. To that end, let us call $w'_i$ the counterparts of the points $v'_i$ points used in the construction of $G_{\xi,\xi',v}$ and define a function $h$ taking value $0$ on the points $v, \kappa, v'_1,\ldots,v'_{t-1}, \xi, w'_1,\ldots,w'_{p-1}$ and $1$ on $\kappa'$ and $\xi'$. The subdivision $S_h$ is a convex subdivision of $Q$ (see Definition \ref{def:subdiv}). Moreover, if we call $\nu_h : Q \rightarrow \R$ the associated convex function, $\nu_h$ is affine linear along $[\kappa',v]$ and since for $i=1,\ldots,t-1$, the angles $v'_{i-1}v'_iv'_{i+1}$ facing $\kappa'$ are all greater than $\pi$, it follows that the subdivision associated to $\nu$ contains the part of the graph $G_{\kappa,\kappa',v}$ in $Q$. In the same manner, this subdivision contains the part of the graph $G_{\xi,\xi',v}$ in $Q$.

 Extend the convex subdivision obtained on $Q$ to the polygon $conv(Q,(0,-1), \linebreak (l-s,-1))$ using Lemma \ref{lem:extend}. The edges $[\kappa,(0,-1)]$ and $[\xi,(l-s,-1)]$ being on the boundary of the polygon, the graph $G$ is still supported by the extended subdivision.

Extend the subdivision once again to the convex polygon $conv(v,\kappa',(-1,0), \linebreak (0,-1),(l-s,-1),(l+1,0),\xi')$. The subdivision we obtain does not necessarily support the graph $G$, the only possibly missing edges being $[(-1,0),\kappa]$, $[(-1,0),\kappa']$, $[(l+1,0),\xi]$ and $[(l+1,0),\xi']$. However, once we refine it using Lemma \ref{lem:refinement}, the unimodular convex subdivision of $conv(v,\kappa',(-1,0),(0,-1),(l-s,-1),(l+1,0),\xi')$ necessarily contains the whole graph $G$. Extending this last subdivision to the whole $\Delta$ and refining it, we conclude that $G$ is supported by a unimodular convex subdivision of $\Delta$.

Now, let us assume that there is a bridge $\sigma'$ ending at $\kappa'$ such that $\tau_{\sigma'}$ is in $\im(\mu)$ (resp. $[\tau_{\sigma'}]$ is in $\im([\mu])$). Using Lemma \ref{lem:allbridge} we know that $\tau_{\rho_1}$ is in $\im(\mu)$ (resp. $[\tau_{\rho_1}]$ is in $\im([\mu])$). 

We first show that $\tau_{\sigma_1}$ is in $\im(\mu)$ (resp. $[\tau_{\sigma_1}]$ is in $\im([\mu])$). Let us call $\eta_1$ and $\eta_2$ the two edges of $G_{\xi,\xi',v}$ going out of $v$. Since $\eta_1$ and $\eta_2$ generate the lattice, take $m_1$ and $m_2$ such that $\sigma_1 + m_1\eta_1 + m_2 \eta_2 = 0$ (the vectors being oriented going out of $v$). Let $m : E(G)\rightarrow \Z$ be the sum of $m^{1,0}_{\kappa,\kappa',v}$ and $m^{m_1,m_2}_{\xi,\xi',v}$ (both extended by $0$ where they are not defined) and define the weighted graph $\G = (G,m)$. It is admissible so from Theorem \ref{thm:rea} we know that $\tau_{\G}\in\im(\mu)$. From Proposition \ref{prop:side}, we have $\tau_{\rho_2}\in\im(\mu)$. By assumption, we also know that $\tau_{\rho_1}\in\im(\mu)$ (resp. $[\tau_{\rho_1}]\in\im([\mu])$). Applying Lemma \ref{lem:chasing} multiple times, we obtain that $\tau_{\sigma_1}$ is in $\im(\mu)$ (resp. $[\tau_{\sigma_1}]$ is in $\im([\mu])$).

We now show that $\tau_{\sigma_2}$ is in $\im(\mu)$ (resp. $[\tau_{\sigma_2}]$ is in $\im([\mu])$). We proceed by induction on the number $t$ of steps required to construct $G_{\kappa,\kappa',v}$. Let us first assume that $t=1$. Take $m_1$ and $m_2$ such that $\sigma_2 + m_1\eta_1 + m_2 \eta_2 = 0$ (the vectors being oriented going out of $v$). Let $m : E(G)\rightarrow \Z$ be the sum of $m^{0,1}_{\kappa,\kappa',v}$ and $m^{m_1,m_2}_{\xi,\xi',v}$ and define the weighted graph $\G=(G,m)$. It is admissible so from Theorem \ref{thm:rea} we know that $\tau_{\G}\in\im(\mu)$. From Proposition \ref{prop:corner}, we have $\tau_{\rho_3},\tau_{\rho_4}\in\im(\mu)$. Applying Lemma \ref{lem:chasing} multiple times, we obtain that $\tau_{\sigma_2}$ is in $\im(\mu)$ (resp. $[\tau_{\sigma_2}]$ is in $\im([\mu])$).

Assume that $t>1$ and that the result is true for $t-1$. Once again, take $m_1$ and $m_2$ such that $\sigma_2 + m_1\eta_1 + m_2 \eta_2 = 0$ (the vectors being oriented going out of $v$). Let $m : E(G)\rightarrow \Z$ be the sum of $m^{0,1}_{\kappa,\kappa',v}$ and $m^{m_1,m_2}_{\xi,\xi',v}$ and define the weighted graph $\G=(G,m)$. It is admissible so from from Theorem \ref{thm:rea} we know that $\tau_{\G}\in\im(\mu)$. Since the graph $G_{\kappa,\kappa',v'_1}$ is constructed in $t-1$ steps we use the induction assumption and the first part of the Lemma to deduce that $\tau_{\sigma_{1,1}}, \tau_{\sigma_{2,1}}\in\im(\mu)$ (resp. $[\tau_{\sigma_{1,1}}], [\tau_{\sigma_{2,1}}]\in\im([\mu])$). Applying Lemma \ref{lem:chasing}, we obtain that $\tau_{\sigma_2}$ is in $\im(\mu)$ (resp. $[\tau_{\sigma_2}]$ is in $\im([\mu])$).

We conclude by induction on $t$ that $\tau_{\sigma_2}$ is in $\im(\mu)$ (resp. $[\tau_{\sigma_2}]$ is in $\im([\mu])$).

The case of the graph $\G_{\kappa',\kappa,v}^{m_1,m_2}$ is very similar. The graph $G$ differs a little to ensure that it is supported by a convex subdivision. Indeed, there are two cases.

\begin{itemize}
\item The first one is if $\kappa'$ is not itself a vertex of $\da$. Then we denote by $\xi$ the vertex of $\da$ next to $\kappa$ such that $\kappa'$ is on $[\xi,\kappa]$ and take $\xi'$ to be the boundary point of $\da$ next to $\xi$ and outside $[\xi,\kappa]$. In the coordinates chosen previously, $\xi$ is of the form $(0,l)$ and $\xi'$ is given by $(1,l+s)$, for some integers $l$ and $s$. The points $(-1,l-s)$ and $(0,l+1)$ are boundary lattice points of $\Delta$.

\item In the second case, $\kappa'$ is a vertex of $\da$ and we take $\xi\neq\kappa$ the vertex of $\da$ next to $\kappa'$ and $\xi'$ the boundary point of $\da$ next to $\xi$ and not in $[\xi, \kappa']$.
 \end{itemize}

In both cases, the graph $G$ is the union of $G_{\kappa',\kappa,v}$ and $G_{\xi,\xi',v}$. The rest of the proof follows the same lines as for $\G_{\kappa,\kappa',v}^{m_1,m_2}$ so we omit it.
\end{proof}

We can now use the graphs we constructed to show the following proposition which provides us with a lot of Dehn twists in the image of the monodromy.

\begin{Proposition}\label{prop:interior}
  Let $\sigma'$ be a bridge ending at $\kappa'$. Assume that $\tau_{\sigma'}\in\im(\mu)$ (resp. $[\tau_{\sigma'}]\in\im([\mu])$). Then for any primitive integer segment $\sigma$ joining two points of $\Delta\cap \Z^2$ not both on the boundary of $\Delta$, we have $\tau_{\sigma}\in\im(\mu)$ (resp. $[\tau_{\sigma}]\in\im([\mu])$).
\end{Proposition}

\begin{proof}
  Let us first note that using Lemma \ref{lem:propagate}, we may move the vertex $\kappa$ without changing the assumptions of the Proposition.

Let $v$ and $w$ in $\Delta\cap\Z^2$ be the two ends of any primitive integer segment $\sigma$. Let us first assume that both $v$ and $w$ are in $\itr(\da)\cap\Z^2$. Up to renaming $v$ and $w$, we may choose two consecutive vertices  $\kappa$ and $\xi$ of $\da$ and boundary points $\kappa'$ on $[\kappa,\xi]$ next to $\kappa$ and $\xi'$ next to $\xi$ outside of $[\kappa,\xi]$ such that the pentagon $\kappa,v,w,\xi',\xi$ is convex with its vertices arranged in this order. It can happen that $\kappa' = \xi$. Consider the graph $G$ to be the reunion of the graphs $G_{\kappa',\kappa,v}$ and $G_{\xi,\xi',w}$ and the segment $\sigma$ (see Figure \ref{fig:biggraphg}). It can be shown as in the proof of Lemma \ref{lem:diamonds} that $G$ is supported by a unimodular convex subdivision.

\begin{figure}[h]
\centering{
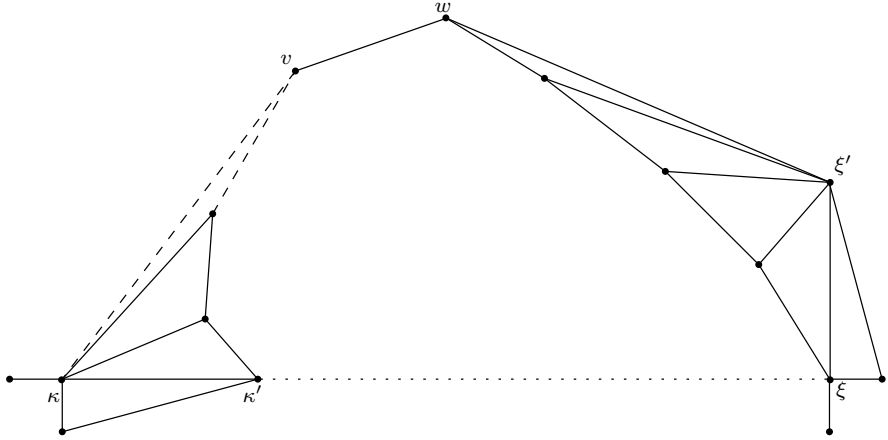}
\caption{Rough sketch of the graph $G$. There is no lattice point in the small triangles. The dashed lines indicate the edges of $G$ for which we know that the associated Dehn twist is in the image of the monodromy.} \label{fig:biggraphg}
\end{figure}

Let $\sigma_1$ and $\sigma_2$ (resp. $\eta_1$ and $\eta_2$) be the primitive integer segments in $G_{\kappa',\kappa,v}$ (resp. in $G_{\xi,\xi',w}$) going out of $v$ (resp. out of $w$). Since $\sigma_1$ and $\sigma_2$ generate the lattice, choose $m_1$ and $m_2$ such that $\sigma+ m_1\sigma_1+m_2\sigma_2 = 0$ (where the orientations go out of $v$). In the same manner, choose $m'_1$ and $m'_2$ such that $\sigma+ m'_1\eta_1+m'_2\eta_2 = 0$ (where the orientations go out of $w$). Let $m: E(G)\rightarrow \Z$ be the sum of $m^{m_1,m_2}_{\kappa',\kappa,v}$, $m^{m'_1,m'_2}_{\xi,\xi',w}$ and the function equal to $1$ on $\sigma$. Define the weighted graph $\G = (G,m)$. It is admissible so from Theorem \ref{thm:rea}, we know that $\tau_{\G}\in\im(\mu)$.

 Assume that $\tau_{\sigma'}\in\im(\mu)$ (resp. $[\tau_{\sigma'}]\in\im([\mu])$). We know from Lemma \ref{lem:diamonds} that $\tau_{\sigma_1},\tau_{\sigma_2}\in\im(\mu)$ (resp. $[\tau_{\sigma_1}],[\tau_{\sigma_2}]\in\im([\mu])$). Using Lemma \ref{lem:chasing} once, we obtain that $\tau_{\sigma}\in\im(\mu)$ (resp. $[\tau_{\sigma}]\in\im([\mu])$).

If one of the two ends of $\sigma$, say $v$ is on the boundary of $\da$, take $\xi\neq v$ to be an end of the edge of $\da$ containing $v$ and $\xi'$ a boundary point of $\da$ next to $\xi$ and outside this face. In the construction of the graph $G$, we replace the graph $G_{\kappa,\kappa',v}$ by a graph consisting of the ray starting at $v$ and going through $\xi$ and any bridge ending at $v$.

By Proposition \ref{prop:side}, we can assume that the ends of $\sigma$ are not on the same edge of $\da$. Take $\kappa \neq v,w$ and $\xi\neq v,w$ to be the ends of the edges of $\da$ containing respectively $v$ and $w$ and such that $\kappa$ and $\xi$ are on the same side of the line joining $v$ and $w$. It can happen that $\kappa = \xi$. In the construction of the graph $G$, we replace the graphs $G_{\kappa,\kappa',v}$ and $G_{\xi,\xi',w}$  respectively by a graph consisting of the ray starting at $v$ and going through $\kappa$ and any bridge ending at $v$ and the graph consisting of the ray starting at $w$ and going through $\xi$ and any bridge ending at $w$.

If one of the two ends of $\sigma$, say $v$ is on the boundary of $\Delta$, we do not need to balance the graph at $v$ so we remove the part of $G$ coming out of $v$ in the preceding construction.

In all those cases, one can check that the new graph $G$ is still supported by a unimodular convex subdivision and that it can be balanced with a weight $1$ on $\sigma$ to obtain an admissible weighted graph $\G$. The rest of the argument follows as in the first case.

Once again, the proof of the homological statement is the same as the one above so we so we omit it.
\end{proof}

\subsection{$K_{\Xd}\otimes \Li$ prime}\label{sec:prime}

In this section,  we prove the second case of Theorem \ref{thm:main1}, that is when $d=2$ and $n=1$. One of the corner-stones of the proof is the following.

\begin{Proposition}\label{prop:gcdedges}
Let $s := \gcd \big( \left\lbrace l_\epsilon \, \big| \, \forall \, \epsilon \text{ edge of } \da \right\rbrace \big)$. Then for any bridge $\sigma \subset \Delta$, we have $(\tau_\sigma)^{s} \in \im(\mu)$ .
\end{Proposition}

In order to prove the latter statement, we will need the following lemmas.

\begin{Lemma}\label{lem:gcd1}
Let $\kappa$ be a vertex of $\da$ and $\epsilon_1, \epsilon_2 \subset \da$ the edges adjacent to $\kappa$. Take $\sigma \subset \Delta$ to be a bridge ending at the point of $\epsilon_1$ at integer distance $m$ of $\kappa$ and $\sigma' \subset \Delta$ to be a bridge ending at the point of $\epsilon_2$ at integer distance $l$ of $\kappa$, for some $m,l\geq 1$.

 Assume that $\tau_{\sigma}\in\im(\mu)$. Then we have $(\tau_{\sigma'})^{\frac{m}{m\wedge l}} \in \im (\mu)$.
\end{Lemma}

\begin{proof}
Using a normalization of $\Delta$ at $\kappa$ if necessary, we can assume that $\kappa=(0,0)$, $\epsilon_1$ is supported by the horizontal axis and that $\sigma$ ends at $(m,0)$. By Lemma \ref{lem:allbridge}, we can assume without loss of generality that $\sigma$ starts at $(-1,0)$. Consider the admissible graph $\mathcal{G}$ of Figure \ref{fig:gcd1} where the vertical edges have weight $-(l\frac{m}{m\wedge l}+\frac{l}{m\wedge l})$ and the horizontal ones have weight $-(m\frac{l}{m\wedge l} + \frac{m}{m\wedge l})$.

\begin{figure}[h]
\centering
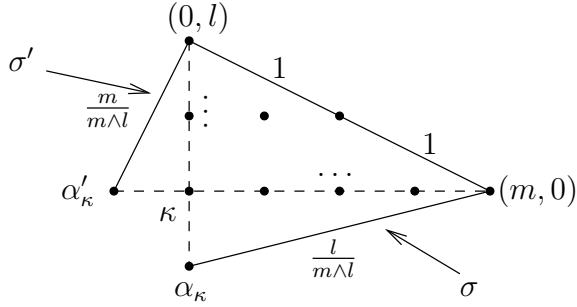
\caption{The admissible graph $\mathcal{G}$.}
\label{fig:gcd1}
\end{figure}

As $\tau_\sigma \in \im(\mu)$, so is $(\tau_\sigma)^\frac{l}{m\wedge l}$. By Theorem \ref{thm:rea}, we have $\tau_\mathcal{G} \in \im(\mu)$. Now using Proposition \ref{prop:side} and Lemma \ref{lem:chasing}, we can chase any primitive integer segment of $\tau_\mathcal{G} \circ(\tau_\sigma)^{-\frac{l}{m\wedge l}} \in \im(\mu)$ except $\sigma'$ and deduce that $(\tau_{\sigma'})^{\frac{m}{m\wedge l}} \in \im(\mu)$.
\end{proof}

\begin{Lemma}\label{lem:gcd2}
Let $\kappa$ be a vertex of $\da$, $\epsilon_1$ and $\epsilon_2$ be the two edges of $\da$ adjacent to $\kappa$ and let $m \leq \min \{ l_{\epsilon_1}, l_{\epsilon_2} \}$ be a positive integer. Let $\sigma \subset \Delta$ be a bridge ending at the point next to $\kappa$ on $\epsilon_1$. If $(\tau_{\sigma})^m \in \im (\mu)$, then $(\tau_{\sigma'})^m \in \im (\mu)$ for any bridge $\sigma' \subset \Delta$ ending on $\epsilon_2$.
\end{Lemma}

\begin{proof}
Using a normalization of $\Delta$ at $\kappa$ if necessary, we can assume that $\kappa=(0,0)$ and that $\epsilon_1$ is supported by the vertical axis. By Lemma \ref{lem:allbridge}, we can assume without loss of generality that $\sigma$ joins $(-1,0)$ to $(0,1)$.
Consider the admissible graphs $\mathcal{G}$  and $\mathcal{G}'$ of Figure \ref{fig:gcd2} where the vertical edges have weight $-m-1$ and the horizontal ones have weight $-2m$ for $\G$ and $-m-1$ for $\G'$. By Theorem \ref{thm:rea}, we have $\tau_\mathcal{G}, \, \tau_{\mathcal{G}'} \in \im(\mu)$. Using Proposition \ref{prop:side} and Lemma \ref{lem:chasing}, we can chase any primitive integer segment of $\tau_\mathcal{G} \circ (\tau_\sigma)^{-1} \in \im(\mu)$ except $\sigma'$ and deduce that $\tau_{\sigma'} \in \im(\mu)$. Using again Proposition \ref{prop:side} and Lemma \ref{lem:chasing}, we can chase any primitive integer segment of $\tau_{\mathcal{G}'} \circ (\tau_{\sigma'})^{-1} \in \im(\mu)$ except $\sigma''$ and deduce that $\tau_{\sigma''} \in \im(\mu)$. To conclude, we can apply the Lemma \ref{lem:gcd1} to the vertex $\kappa$ and the primitive integer segment $\sigma''$.
 
\begin{figure}[h]
\centering
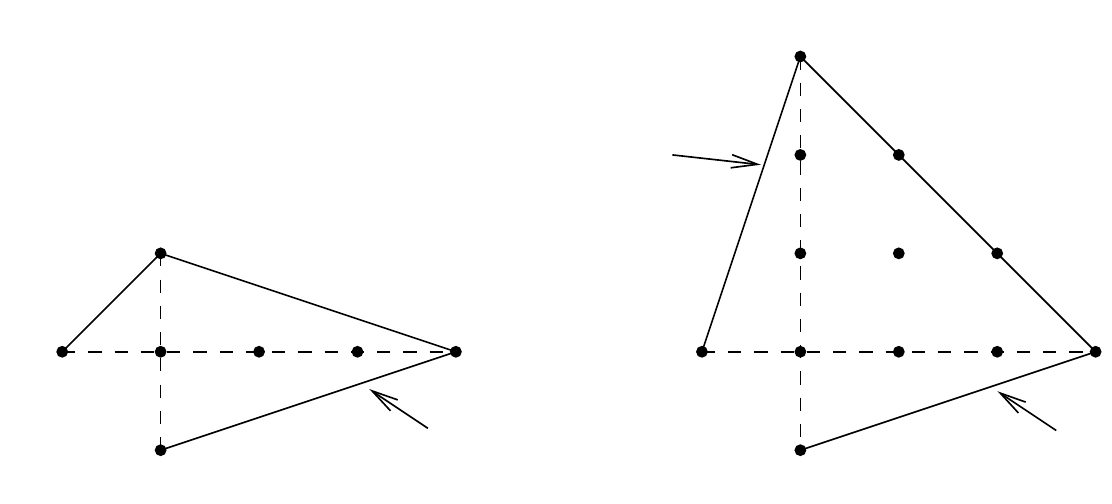
\caption{The admissible graphs $\G$ (left) and $\G'$ (right).}
\label{fig:gcd2}
\end{figure}

\end{proof}

\newpage

\begin{proof}[Proof of Proposition \ref{prop:gcdedges}]
Let $\epsilon_1$ and $\epsilon_2$ be two consecutive edges of $\da$ intersecting at the vertex $\kappa$.  We first claim that for any bridge $\sigma$ ending on $\epsilon_1\cup\epsilon_2$, we have $(\tau_{\sigma})^{l_{\epsilon_1}\wedge l_{\epsilon_2}}\in\im(\mu)$. Without loss of generality, we may assume that $l_{\epsilon_1} \geq l_{\epsilon_2}$ and consider the normalization of $\Delta$ at $\kappa$ such that $\epsilon_1$ is horizontal. Let $\xi$ be the other end of $\epsilon_2$.

Since $\xi$ is a vertex, we know from Proposition \ref{prop:corner} that if $\sigma'$ is a bridge ending at $\xi$ then $\tau_{\sigma'}\in\im(\mu)$. Using Lemma \ref{lem:gcd1} and then Lemma \ref{lem:gcd2}, we deduce that for any bridge $\sigma$ ending on $\epsilon_1\cup\epsilon_2$, we have $(\tau_{\sigma})^{l_{\epsilon_2}}\in\im(\mu)$.

Now consider the weighted graph $\G''$ given in Figure \ref{fig:gcdedges} where the horizontal edges have weight $-2l_{\epsilon_1}$ and the vertical ones have weight $-2$.

\begin{figure}[h]
\centering
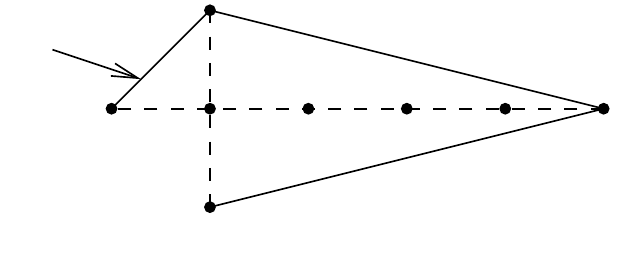
\caption{The admissible graph $\mathcal{G}''$.}
\label{fig:gcdedges}
\end{figure}

This graph is admissible, so by Theorem \ref{thm:rea}, we know that $\tau_{\G''}\in\im(\mu)$. Using Propositions \ref{prop:corner} and \ref{prop:side} as well as Lemma \ref{lem:chasing}, we obtain that $(\tau_{\sigma''})^{l_{\epsilon_1}}\in\im(\mu)$ and thus $(\tau_{\sigma''})^{l_{\epsilon_1}\wedge l_{\epsilon_2}}\in\im(\mu)$. Applying Lemma \ref{lem:gcd2} twice, we deduce that for any bridge $\sigma$ ending on $\epsilon_1\cup\epsilon_2$, we have $(\tau_{\sigma})^{l_{\epsilon_1}\wedge l_{\epsilon_2}}\in\im(\mu)$ which proves the claim.

Using this claim and Lemma \ref{lem:gcd2}, we show by induction that for any $k$ consecutive edges $\epsilon_1$, ...,$\epsilon_k$ of $\da$, we have $(\tau_{\sigma})^{l_{\epsilon_1}\wedge...\wedge l_{\epsilon_k}} \in \im(\mu)$ for all bridge $\sigma$ ending on $\epsilon_1 \cup ... \cup\epsilon_k$. The result follows.
\end{proof}

The proof of the case $d=2$ in Theorem \ref{thm:main1} consists in finding a family of Dehn twists in the image of the monodromy which generates the mapping class group. This family will appear as a ``snake'' in $\Delta$.

\begin{Definition}
  A snake in the polygon $\Delta$ is a family of primitive integer segments $\sigma_1 = [v_0,v_1] ,\sigma_2 = [v_1,v_2],\ldots,\sigma_{g_{\Li}} = [v_{g_{\Li}-1},v_{g_{\Li}}], \sigma$ in $\Delta$ such that
  \begin{enumerate}
  \item $v_0$ is a boundary point of $\Delta$ while $v_1$ is a vertex of $\da$ and $v_2,\ldots,v_{g_{\Li}}$ are distinct points of $\da$,
\item the intersection of any two segments is either empty or one of the $v_i$,
\item $\sigma$ is a bridge starting at $v_2$.
  \end{enumerate}
\end{Definition}

Given a snake in $\Delta$, the loops $\delta_{v_i}$ and $\delta_{\sigma_i}$ are arranged as in the Figure \ref{fig:humph}.

\begin{figure}[h]
\centering{
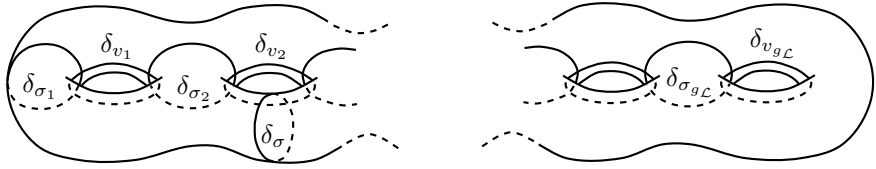}
\caption{The loops $\delta_{\sigma_i}$ and $\delta_{v_i}$ in $C_0$} \label{fig:humph}
\end{figure}

\begin{Lemma}\label{lem:serpent}
 There always exists a snake $\sigma_1 = [v_0,v_1] ,\sigma_2 = [v_1,v_2],\ldots,\sigma_{g_{\Li}} = [v_{g_{\Li}-1},v_{g_{\Li}}], \sigma$ in $\Delta$. 

Moreover, if for all the segments of the snake the associated Dehn twist is in $\im(\mu)$ (resp. in $\im([\mu])$), then the map $\mu$ (resp. $[\mu]$) is surjective.
\end{Lemma}

\begin{proof}
  To prove the existence of a snake, choose a vertex $\kappa$ of $\da$ and take a normalization of $\Delta$ at $\kappa$. Let $v_0 = (-1,0)$, $v_1 = \kappa$, $v_2 = (1,0)$ and $\sigma = [(0,-1),(1,0)]$. We number the rest of the lattice points in $\da$ by going through them in the colexicographic order.

We know from Theorem \ref{thm:acycle} that for $i=1,\ldots,g_{\Li}$, the Dehn twist $\tau_{v_i}$ is in $\im(\mu)$. The second part of the Lemma then follows from Humphries' Theorem \cite{Humphries}, which states that the family $\tau_{\sigma_1},\tau_{v_1},\ldots,\tau_{\sigma_{g_{\Li}}},\tau_{v_{g_{\Li}}}$, and $\tau_{\sigma}$ generates the group $\MCG(C_0)$.
\end{proof}

\begin{proof}[Proof of Theorem \ref{thm:main1}, case $d=2$ and $n=1$]
Assume that the largest order of a root of $K_{\Xd}\otimes\Li$ is $1$. We know from Proposition \ref{prop:rootandlength} that $\gcd \big( \left\lbrace l_\epsilon \, \big| \, \forall \, \epsilon \text{ edge of } \da \right\rbrace \big) = 1$. Thus it follows from Proposition \ref{prop:gcdedges} that for any bridge the associated Dehn twist is in $\im(\mu)$. Proposition \ref{prop:interior} then implies that for any primitive integer segment $\sigma$ joining two points of $\Delta\cap \Z^2$ not both on the boundary of $\Delta$, we have $\tau_{\sigma}\in\im(\mu)$ .

In particular, taking any snake in $\Delta$, all the corresponding Dehn twists are in $\im(\mu)$. Lemma \ref{lem:serpent} then gives the result we wanted.
\end{proof}

\subsection{$K_{\Xd}\otimes \Li$ odd}\label{sec:odd}

\begin{proof}[Proof of Theorem \ref{thm:main2}, case $d=2$ and $n$ odd]
Since $K_{\Xd}\otimes\Li$ does not have a square root, the largest order of one of its roots is odd. Let us show that the algebraic monodromy $[\mu]$ is surjective on $\spaut(H_1(C_0,\Z))$.

Notice first that since $n$ is odd, we know from Proposition \ref{prop:rootandlength} that the greatest common divisor $s$ of the lengths of the edges of $\da$ is odd.

Using Lemma \ref{lem:serpent}, take a snake $\sigma_1 = [v_0,v_1] ,\sigma_2 = [v_1,v_2],\ldots,\sigma_{g_{\Li}} = [v_{g_{\Li}-1},v_{g_{\Li}}], \sigma$ in $\Delta$. We know from Proposition \ref{prop:gcdedges} that $\tau_{\sigma}^s$ is in $\im(\mu)$.

On the other hand, $[\tau_{\sigma}]^2$ is in $\im([\mu])$. Indeed, we know from Propositions \ref{prop:corner} and \ref{prop:side} and Theorem \ref{thm:acycle} that $[\tau_{\sigma_1}]$, $[\tau_{\sigma_2}]$ and $[\tau_{v_1}]$ are in $\im([\mu])$. Since both boundary components of a closed regular neighborhood of $\delta_{\sigma_1}\cup\delta_{v_1}\cup \delta_{\sigma_2}$ are homologous to $\delta_{\sigma}$ in $C_0$ if one orients them carefully, using the chain rule (see Prop 4.12 \cite{farbmarg}), we have
\[
([\tau_{\sigma_1}][\tau_{v_1}][\tau_{\sigma_2}])^4 = [\tau_{\sigma}]^2\in\im([\mu]).
\]

Since $s$ is odd, we deduce that $[\tau_{\sigma}]$ is in $\im([\mu])$.
 
We apply Proposition \ref{prop:interior} to deduce that for any $i=1,\ldots,g_{\Li}$, $[\tau_{\sigma_i}]$ is in $\im([\mu])$. Lemma \ref{lem:serpent} provides the last piece of the proof.
\end{proof}

\subsection{About the Conjecture \ref{conj:general}}\label{sec:conj}

We show in this section that when the adjoint bundle of $\Li$ admits a non-trivial root $\LS$, we can still exhibit a lot of (powers of) Dehn twists in $\im(\mu)$ and $\im(\mmu)$. This is a motivation for Conjecture \ref{conj:general} as well as a preparation for the results of \cite{article2}. Indeed, in \cite{article2} we will show that the conjecture is true when the highest order of a root of $K_{\Xd}\otimes\Li$ is $2$.

Let $d\geq 1$ be an integer and $h_{\frac{1}{d},\kappa}$ be the homothety of ratio $\frac{1}{d}$ and centered at $\kappa$. When $h_{\frac{1}{d},\kappa}(\da)$ is again a lattice polygon, we say that $\da$ is divisible by $d$ and we denote by $\da(d)$ the subset $h_{\frac{1}{d},\kappa} ^{-1}(h_{\frac{1}{d},\kappa}(\da)\cap \Z)$. Note that neither $d$ nor $\da(d)$ depend on the choice of the vertex $\kappa$. In this case, we introduce a variation of the graphs $\G_{\kappa,\kappa',v}^{m_1,m_2}$ and  $\G_{\kappa',\kappa,v}^{m_1,m_2}$ which we will use to motivate conjecture \ref{conj:general}.

Assume that $v = h_{\frac{1}{d},\kappa} ^{-1}(u)\in\da(d)$ and denote by $\G_{\kappa,\kappa',v}^{m_1,m_2}(d)$ (resp. $\G_{\kappa',\kappa,v}^{m_1,m_2}(d)$ the following weighted graphs. The graph $G_{\kappa,\kappa',v}(d)$ (resp. $G_{\kappa',\kappa,v}(d)$)  is given as the reunion of $h_{\frac{1}{d},\kappa} ^{-1}(G_{\kappa,\kappa',u}\cap\da)$ (resp. of $h_{\frac{1}{d},\kappa} ^{-1}(G_{\kappa',\kappa,u}\cap\da)$) and of the segments $[\akp, h_{\frac{1}{d},\kappa} ^{-1}(\kappa')]$, $[\akp,\kappa]$ and $[\ak,\kappa]$. The weight function  on $G_{\kappa,\kappa',v}(d)$ (resp. on $G_{\kappa',\kappa,v}(d)$) is equal to $m_{\kappa,\kappa',u}^{m_1,m_2}\circ h_{\frac{1}{d},\kappa}$ (resp. to $m_{\kappa',\kappa,u}^{m_1,m_2}\circ h_{\frac{1}{d},\kappa}$) inside $\da$ and adjusted on the last three segments to be balanced at $\kappa$ and $h_{\frac{1}{d},\kappa} ^{-1}(\kappa')$. We denote by $\sigma_1$ and $\sigma_2$ the two primitive integer segments in $G_{\kappa,\kappa',v}(d)$ (resp. in $G_{\kappa',\kappa,v}(d)$) which have an end at $v$.

 We have the analogue of Lemma \ref{lem:diamonds}. We omit its proof as it is the same as that of Lemma \ref{lem:diamonds}.

\begin{Lemma}\label{lem:diamondsd}
Let $d\geq 1$ be an integer and assume that $\da$ is divisible by $d$. Let $v\in\da(d)$.\\
 The weighted graphs $\G_{\kappa,\kappa',v}^{m_1,m_2}(d)$ and $\G_{\kappa',\kappa,v}^{m_1,m_2}(d)$ satisfy the following properties:
  \begin{itemize}
  \item the edges $\sigma_1 = \sigma_{1,0}$ and $\sigma_2 = \sigma_{2,0}$ generate the lattice,
  \item the weight function $m_{\kappa,\kappa',v}^{m_1,m_2}$ (resp. $m_{\kappa',\kappa,v}^{m_1,m_2}$) satisfies the balancing condition \eqref{eq:balance} at every element of $V(G_{\kappa,\kappa',v}(d))\cap\itr\Delta$ (resp. $V(G_{\kappa',\kappa,v}(d))\cap\itr\Delta$) except at $v$.
  \end{itemize}
Assume moreover that there is a bridge $\sigma'$ ending at $h_{\frac{1}{d},\kappa} ^{-1}(\kappa')$ such that $\tau_{\sigma'}$ is in $\im(\mu)$ (resp. $[\tau_{\sigma'}]$ is in $\im([\mu])$). Then $\tau_{\sigma_1}$ and $\tau_{\sigma_2}$ are in $\im(\mu)$ (resp. $[\tau_{\sigma_1}]$ and $[\tau_{\sigma_2}]$ are in $\im([\mu])$). \qed
\end{Lemma}

\begin{Proposition}\label{prop:interiord}
 Let $d\geq 1$ be an integer and assume that $\da$ is divisible by $d$.  Assume that $\tau_{\sigma'}\in\im(\mu)$ (resp. $[\tau_{\sigma'}]\in\im([\mu])$) for any bridge $\sigma'$ ending at a point in $\da(d)$. Then for any primitive integer segment $\sigma\subset \Delta$, if the line it generates intersects $\da(d)$ then we have $\tau_{\sigma}\in\im(\mu)$ (resp. $[\tau_{\sigma}]\in\im([\mu])$).
\end{Proposition}

\begin{proof}
Let us first assume that the line $L$ generated by $\sigma$ intersects $\da(d)$ in a point $u$. The proof uses the same ideas as that of Proposition \ref{prop:interior}. Namely, let $v$ and $w$ be the two ends of $\sigma$. We may assume that $v$ is on the segment $[u,w]$ and that $[u,w]\cap \da(d) = \{u\}$. We will only give the proof in the case where both $u$ and $w$ are in the interior of $\da$ as the other cases follow the same pattern as in Proposition \ref{prop:interior}.

To that end, choose $\kappa$ and $\xi$ two consecutive vertices  of $\da$ which are on the same side of $L$ and such that the points $u$, $w$, $\xi$ and $\kappa$ appear in this order when following the boundary of their convex hull in the clockwise orientation. Take a normalization of $\Delta$ at $\kappa$ such that the edge $[\kappa,\xi]$ is horizontal.

Since the roles of $v$ and $w$ are not the same, we will need to consider two cases. If the line $L$ intersects the vertical line though $\kappa$ in the interval $[\kappa,(0,d))$, take $\kappa'=(1,0)$ and $\xi'$ the boundary lattice point of $\da$ adjacent to $\xi$ and not on the segment $[\kappa,\xi]$. In the other case, take $\kappa'=(0,1)$ and $\xi'$ the boundary point of $\da$ adjacent to $\xi$ and on the segment $[\kappa,\xi]$. Consider now the graph $G$ to be the reunion of $G_{\kappa',\kappa,u}(d)$, $[u,w]$ and $G_{\xi,\xi',w}$ in the first case and the reunion of $G_{\kappa,\kappa',u}(d)$, $[u,w]$ and $G_{\xi',\xi,w}$ in the second case (see Figure \ref{fig:biggrapheven}). Notice that in the first case, the pentagon $\kappa, u, w, \xi', \xi$ is convex and in the second case, the pentagon $\kappa, h_{\frac{1}{d},\kappa}^{-1}(\kappa'), u, w, \xi$ is convex. It can be shown as in the proof of Proposition \ref{prop:interior} that $G$ is supported by a unimodular convex subdivision.

\begin{figure}[h]
\centering
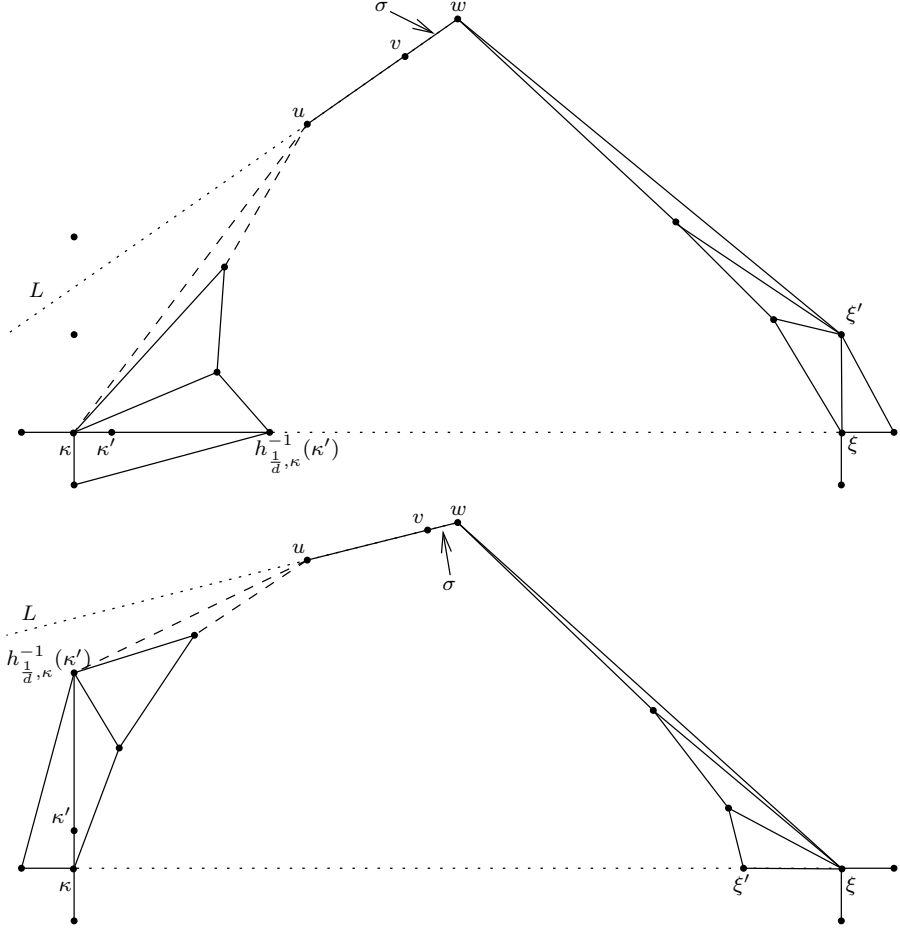
\caption{Rough sketch of the admissible graph $\G$ in the two cases.}
\label{fig:biggrapheven}
\end{figure}

We next define a weighted graph $\G = (G,m)$. Let $\sigma_1$ and $\sigma_2$ (resp. $\eta_1$ and $\eta_2$) be the primitive integer segments in $G$ going out of $u$ (resp. out of $w$). Since $\sigma_1$ and $\sigma_2$ generate the lattice, choose $m_1$ and $m_2$ such that $\sigma+ m_1\sigma_1+m_2\sigma_2 = 0$ (where the orientations go out of $u$). In the same manner, choose $m'_1$ and $m'_2$ such that $\sigma+ m'_1\eta_1+m'_2\eta_2 = 0$ (where the orientations go out of $w$). In the first case, take the weight function $m$ to be that of $\G^{m_1,m_2}_{\kappa',\kappa,u}(d)$ on $G_{\kappa',\kappa,u}(d)$ and that of $\G^{m'_1,m'_2}_{\xi,\xi',w}$ on $G_{\xi,\xi',w}$ and equal to $1$ on $[u,w]$. In the second case, we take the weight function of $\G^{m_1,m_2}_{\kappa,\kappa',u}(d)$ on $G_{\kappa,\kappa',u}(d)$ and that of $\G^{m'_1,m'_2}_{\xi',\xi,w}$ on $G_{\xi',\xi,w}$ and equal to $1$ on $[u,w]$. The weighted graph $\G$ is admissible so from Theorem \ref{thm:rea}, we know that $\tau_{\G}\in\im(\mu)$. The conclusion is the same as in Proposition \ref{prop:interior}.

The cases where at least one of the ends of $[u,w]$ is not in $\itr(\Delta)$ can be treated in the same way as in Proposition \ref{prop:interior}.
\end{proof}

As we noticed in Corollary \ref{cor:obstruction}, when $\da$ is divisible by an integer $d>1$, we cannot have $\tau_{\sigma}\in\im(\mu)$ for all primitive integer segments in $\Delta$. However, we have the following two statements.

\begin{Lemma}\label{lem:diad}
 Let $d> 1$ be an integer and assume that $\da$ is divisible by $d$.  Assume that $\tau_{\sigma'}\in\im(\mu)$ (resp. $[\tau_{\sigma'}]\in\im([\mu])$) for any bridge $\sigma'$ ending at a point in $\da(d)$. Let $w\in \da$ be an interior lattice point, $\kappa$ any vertex of $\da$ and $\kappa'$ and $\kappa''$ the lattice points of $\partial \da$ next to $\kappa$ on each side. Then for any integers $m_1$ and $m_2$,
\[
(\tau_{\G_{\kappa,\kappa',w}^{m_1,m_2}})^d,(\tau_{\G_{\kappa',\kappa,w}^{m_1,m_2}})^d,(\tau_{\G_{\kappa,\kappa'',w}^{m_1,m_2}})^d,(\tau_{\G_{\kappa'',\kappa,w}^{m_1,m_2}})^d \in \im(\mu).
\]
\end{Lemma}

We prove this Lemma after the following proposition.

\begin{Proposition}\label{prop:interiordd}
Let $d> 1$ be an integer and assume that $\da$ is divisible by $d$.  Assume that $\tau_{\sigma'}\in\im(\mu)$ (resp. $[\tau_{\sigma'}]\in\im([\mu])$) for any bridge $\sigma'$ ending at a point in $\da(d)$. Then for any primitive integer segment $\sigma\subset \Delta$, we have $\tau_{\sigma}^d\in\im(\mu)$ (resp. $[\tau_{\sigma}]^d\in\im([\mu])$).
\end{Proposition}

\begin{proof}
  The idea of the proof is similar to that of Proposition \ref{prop:interior}. Namely, if $v$ and $w$ denote the two ends of $\sigma$, we first assume that they are both in the interior of $\da$ and we consider exactly the same weighted graph $\G = (G,m)$ as in Proposition \ref{prop:interior}. Recall that it contains two subgraphs $\G_1 =\G_{\kappa',\kappa,v}^{m_1,m_2}$ and $\G_2 =\G_{\xi,\xi',w}^{m_1',m_2'}$. From Theorem \ref{thm:rea}, we know that $\tau_{\G_1}\tau_{\G_2}\tau_{\sigma}\in\im(\mu)$. Lemma \ref{lem:diad} gives us the conclusion we want.

The cases where one of the ends is not in the interior of $\da$ can be proved in the same way. We leave it to the reader.
\end{proof}

\begin{proof}[Proof of Lemma \ref{lem:diad}]
Recall that $\sigma_1$ and $\sigma_2$ are the two primitive integer segments in $\G_{\kappa,\kappa',w}^{m_1,m_2}$ going out of $w$. Let us first show that if $\tau_{\G_{\kappa,\kappa',w}^{m_1,m_2}}\in\im(\mu)$ then for any vertex $\xi$ of $\da$,
\[
\tau_{\G_{\xi,\xi',w}^{a_1,a_2}}, \tau_{\G_{\xi,\xi'',w}^{b_1,b_2}}, \tau_{\G_{\xi',\xi,w}^{c_1,c_2}}, \tau_{\G_{\xi'',\xi,w}^{d_1,d_2}} \in\im(\mu),
\]
where the weights are chosen so that adding $m_1\sigma_1+m_2\sigma_2$ balances each graph at $w$.

To that end, we can assume that $\xi$ is a vertex next to $\kappa$, so that $\kappa''$ and $\xi''$ are on $[\kappa,\xi]$.
\begin{itemize}
\item The weighted graph obtained by concatenating $\G_{\kappa,\kappa',w}^{m_1,m_2}$ and $\G_{\xi,\xi',w}^{a_1,a_2}$ is admissible, so by Theorem \ref{thm:rea}, $\tau_{\G_{\kappa,\kappa',w}^{m_1,m_2}}\tau_{\G_{\xi,\xi',w}^{a_1,a_2}}\in\im(\mu)$.
\item The weighted graph obtained by concatenating $\G_{\kappa,\kappa',w}^{m_1,m_2}$ and $\G_{\xi'',\xi,w}^{d_1,d_2}$ is admissible, so by Theorem \ref{thm:rea}, $\tau_{\G_{\kappa,\kappa',w}^{m_1,m_2}}\tau_{\G_{\xi'',\xi,w}^{d_1,d_2}}\in\im(\mu)$.
\item The weighted graph obtained by concatenating $\G_{\xi,\xi',w}^{a_1,a_2}$ and $\G_{\xi,\xi'',w}^{b_1,b_2}$ is admissible, so by Theorem \ref{thm:rea}, $\tau_{\G_{\xi,\xi',w}^{a_1,a_2}} \tau_{\G_{\xi,\xi'',w}^{b_1,b_2}}\in\im(\mu)$.
\end{itemize}
To prove that $\tau_{\G_{\xi',\xi,w}^{c_1,c_2}}\in\im(\mu)$, we consider the vertex $\zeta$ next to $\xi$ such that $\xi'$ and $\zeta'$ are on $[\xi,\zeta]$. We choose weights are chosen so that adding $m_1\sigma_1+m_2\sigma_2$ balances the graph $\G_{\zeta',\zeta,w}^{e_1,e_2}$ at  $w$.
\begin{itemize}
\item The weighted graph obtained by concatenating $\G_{\xi,\xi'',w}^{b_1,b_2}$ and $\G_{\zeta',\zeta,w}^{e_1,e_2}$ is admissible, so by Theorem \ref{thm:rea}, $\tau_{\G_{\xi,\xi'',w}^{b_1,b_2}} \tau_{\G_{\zeta',\zeta,w}^{e_1,e_2}}\in\im(\mu)$.
\item The weighted graph obtained by concatenating $\G_{\xi',\xi,w}^{c_1,c_2}$ and $\G_{\zeta',\zeta,w}^{e_1,e_2}$ is admissible, so by Theorem \ref{thm:rea}, $\tau_{\G_{\xi',\xi,w}^{c_1,c_2}} \tau_{\G_{\zeta',\zeta,w}^{e_1,e_2}}\in\im(\mu)$.
\end{itemize}
Thus, we obtain the result we wanted. Notice that the previous claim is still true if we assume that $\tau_{\G_{\kappa',\kappa,w}^{m_1,m_2}}$ in in $\im(\mu)$ instead of $\tau_{\G_{\kappa,\kappa',w}^{m_1,m_2}}$.

  Now, choose a unimodular convex subdivision of $h_{\frac{|}{d},\kappa}(\da)$. The point $h_{\frac{|}{d},\kappa}(w)$ is in at least one of the triangles of this subdivision, possibly on its boundary. Let $x$, $y$ and $z$ be the three vertices of the image of such a triangle by $h_{\frac{|}{d},\kappa}^{-1}$. Denote respectively by $\gamma_1$, $\gamma_2$ and $\gamma_3$ the primitive integer segments on $[x,w]$, $[y,w]$ and $[z,w]$ ending at $w$.

Let us now show that $(\tau_{\G_{\kappa,\kappa',w}^{m_1,m_2}})^d \in \im(\mu)$. We proceed as in the proof of Proposition \ref{prop:interiord}, using $\gamma_1$ in place of $\sigma$. Using the same notations, we obtain that $\tau_{\G_{\xi,\xi',w}^{m_1',m_2'}}\in\im(\mu)$ in the first case, and $\tau_{\G_{\xi',\xi,w}^{m_1',m_2'}}\in\im(\mu)$ in the second case for some vertex $\xi$ of $\da$. Using the claim we proved in the beginning, we have that
\[
\tau_{\G_{\kappa,\kappa',w}^{a_1,a_2}},\tau_{\G_{\kappa',\kappa,w}^{b_1,b_2}},\tau_{\G_{\kappa,\kappa'',w}^{c_1,c_2}},\tau_{\G_{\kappa'',\kappa,w}^{d_1,d_2}}\in \im(\mu),
\]
where the weights are chosen so that adding $\gamma_1$ with weight $1$ balances the graph at $w$.

We do the same for $\gamma_2$ to prove that
\[
\tau_{\G_{\kappa,\kappa',w}^{a'_1,a'_2}},\tau_{\G_{\kappa',\kappa,w}^{b'_1,b'_2}},\tau_{\G_{\kappa,\kappa'',w}^{c'_1,c'_2}},\tau_{\G_{\kappa'',\kappa,w}^{d'_1,d'_2}}\in \im(\mu),
\]
where the weights are chosen so that adding $\gamma_2$ with weight $1$ balances the graph at $w$, and for $\gamma_3$ to obtain that
\[
\tau_{\G_{\kappa,\kappa',w}^{a''_1,a''_2}},\tau_{\G_{\kappa',\kappa,w}^{b''_1,b''_2}},\tau_{\G_{\kappa,\kappa'',w}^{c''_1,c''_2}},\tau_{\G_{\kappa'',\kappa,w}^{d''_1,d''_2}}\in \im(\mu),
\]
where the weights are chosen so that adding $\gamma_3$ with weight $1$ balances the graph at $w$.

Since $d\Z$ is included in the lattice generated by $\gamma_1$, $\gamma_2$ and $\gamma_3$, we can find integer combinations $\alpha_1\gamma_1 + \alpha_2\gamma_2+\alpha_3\gamma_3 + d\sigma_1 = 0$ and $\beta_1\gamma_1 + \beta_2\gamma_2+\beta_3\gamma_3 + d\sigma_2 = 0$. Thus,
\[
\tau_{\G_{\kappa,\kappa',w}^{d,0}} = (\tau_{\G_{\kappa,\kappa',w}^{a_1,a_2}})^{\alpha_1}(\tau_{\G_{\kappa,\kappa',w}^{a'_1,a'_2}})^{\alpha_2}(\tau_{\G_{\kappa,\kappa',w}^{a''_1,a''_2}})^{\alpha_3} \in\im(\mu)
\]
 and
\[
\tau_{\G_{\kappa,\kappa',w}^{0,d}} = (\tau_{\G_{\kappa,\kappa',w}^{a_1,a_2}})^{\beta_1}(\tau_{\G_{\kappa,\kappa',w}^{a'_1,a'_2}})^{\beta_2}(\tau_{\G_{\kappa,\kappa',w}^{a''_1,a''_2}})^{\beta_3} \in\im(\mu)
\]
It implies that
\[ (\tau_{\G_{\kappa,\kappa',w}^{m_1,m_2}})^d = (\tau_{\G_{\kappa,\kappa',w}^{d,0}})^{m_1} (\tau_{\G_{\kappa,\kappa',w}^{0,d}})^{m_2} \in \im(\mu).\]

We conclude the proof of Lemma \ref{lem:diad} using the same argument for the three other graphs.
\end{proof}

In order to be able to apply Propositions \ref{prop:interiord} and \ref{prop:interiordd}, we need to be able to say something about the bridges ending at a point of $\da(d)$.

\begin{Proposition}
Let $s = \gcd \big( \left\lbrace l_\epsilon \, \big| \, \forall \, \epsilon \text{ edge of } \da \right\rbrace \big)$.  If $v\in\da(s)\cap\partial\da$, then for any bridge $\sigma$ ending at $v$, we have $\tau_{\sigma}\in\im(\mu)$.

Moreover, if $s$ is even and $v\in\da(2)$, then for any bridge $\sigma$ ending at $v$, we have $[\tau_\sigma]\in\im(\mmu)$.
\end{Proposition}

\begin{proof}
We may assume that $v$ is not a vertex as this case is taken care of by Proposition \ref{prop:corner}.

Let $v\in\partial\da$ and let $\kappa$ be a vertex of $\da$ next to $v$. Using a normalization at $\kappa$ we may assume that $v$ is of the form $(l,0)$ for some integer $l\geq 1$. Using Lemma \ref{lem:allbridge} we may assume that $\sigma$ joins $v$ to $\ak = (0,-1)$. Now consider the weighted graph $\G$ given in Figure \ref{fig:evenbridge}, where the horizontal edges have weight $-2l$. This graph is admissible so by Theorem \ref{thm:rea}, we know that $\tau_{\G}\in\im(\mu)$.

\begin{figure}[h]
\centering
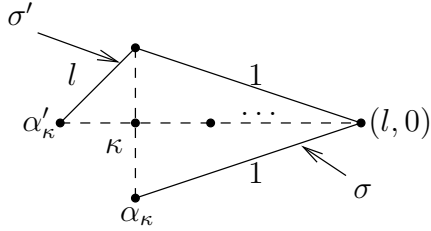
\caption{The weighted graph $\G$}
\label{fig:evenbridge}
\end{figure}

 If $v\in\da(s)$, we have $l = as$ for some integer $a\geq 1$. We know from Proposition \ref{prop:gcdedges} that $\tau_{\sigma'}^{as}\in\im(\mu)$. Thus, applying Propositions \ref{prop:corner} and \ref{prop:side} as well as Lemma \ref{lem:chasing}, we obtain that $\tau_{\sigma}\in\im(\mu)$.

    If $s$ is even and $v\in\da(2)$, then $l = 2a$ for some integer $a\geq 1$. Let $\sigma'$ be the bridge joining $\akp=(-1,0)$ to $(0,1)$. We do not have in general that $\tau_{\sigma'}^2\in\im(\mu)$. However, if we define $\sigma_1=[\ak,\kappa]$ and $\sigma_2=[\kappa,(0,1)]$, the boundary of a regular neighborhood of $\delta_{\sigma_1}\cup\delta_{\kappa}\cup\delta_{\sigma_2}$ has two connected components both homologous to $\delta_{\sigma'}$ if one orients them carefully. Using the chain rule (see Proposition 4.12 in \cite{farbmarg}) we obtain that
\[
([\tau_{\sigma_1}][\tau_{\kappa}][\tau_{\sigma_2}])^4 = [\tau_{\sigma'}]^2.
\]
It follows then from Propositions \ref{prop:corner} and \ref{prop:side} as well as Theorem \ref{thm:acycle} that $[\tau_{\sigma'}]^2\in\im(\mmu)$.

Using the previous construction, we conclude that $[\tau_{\sigma}]\in\im(\mmu)$.
\end{proof}

\bibliographystyle{alpha}
\bibliography{Draft}

\vspace{1cm}
\noindent
\textit{E-mail addresses}: remicretois@yahoo.fr,  lionel.lang@math.uu.se

\end{document}